\documentclass[reqno]{amsart}

\usepackage{enumitem}
\usepackage{amssymb}
\usepackage{hyperref}

\newcommand*{\mailto}[1]{\href{mailto:#1}{\nolinkurl{#1}}}
\newcommand{\arxiv}[1]{\href{http://arxiv.org/abs/#1}{arXiv:#1}}

\newtheorem{theorem}{Theorem}[section]
\newtheorem{definition}[theorem]{Definition}
\newtheorem{lemma}[theorem]{Lemma}

\newtheorem{proposition}[theorem]{Proposition}
\newtheorem{corollary}[theorem]{Corollary}

\newtheorem{remark}[theorem]{Remark}

\newcommand{\R}{{\mathbb R}}

\newcommand{\Z}{{\mathbb Z}}
\newcommand{\C}{{\mathbb C}}

\newcommand{\I}{\mathrm{i}}
\newcommand{\E}{\mathrm{e}}
\DeclareMathOperator{\re}{Re}
\DeclareMathOperator{\im}{Im}
\newcommand{\bK}{{\mathbf K}}
\newcommand{\bM}{{\mathbf M}}

\newcommand{\ve}{\varepsilon}
\newcommand{\calR}{\mathcal{R}}
\newcommand{\cM}{\mathcal{M}}
\newcommand{\floor}[1]{\lfloor#1 \rfloor}
\newcommand{\ceil}[1]{\lceil#1 \rceil}

\numberwithin{equation}{section}

\begin{document}

\title[Dispersion Estimates for Discrete Equations]{Dispersion Estimates for One-Dimensional Discrete Schr\"odinger and Wave Equations}

\author[I. Egorova]{Iryna Egorova}
\address{Faculty of Mathematics\\ University of Vienna\\
Oskar-Morgenstern-Platz 1\\ 1090 Wien\\ Austria\\ and Institute for Low Temperature Physics\\ 47, Lenin ave\\ 61103 Kharkiv\\ Ukraine}
\email{\href{mailto:iraegorova@gmail.com}{iraegorova@gmail.com}}

\author[E.\ Kopylova]{Elena Kopylova}
\address{Faculty of Mathematics\\ University of Vienna\\
Oskar-Morgenstern-Platz 1\\ 1090 Wien\\ Austria\\ and Institute for Information Transmission Problems\\ Russian Academy of Sciences\\
Moscow 127994\\ Russia}
\email{\mailto{Elena.Kopylova@univie.ac.at}}
\urladdr{http://www.mat.univie.ac.at/~ek/}

\author[G.\ Teschl]{Gerald Teschl}
\address{Faculty of Mathematics\\ University of Vienna\\
Oskar-Morgenstern-Platz 1\\ 1090 Wien\\ Austria\\ and International Erwin Schr\"odinger
Institute for Mathematical Physics\\ Boltzmanngasse 9\\ 1090 Wien\\ Austria}
\email{\mailto{Gerald.Teschl@univie.ac.at}}
\urladdr{\url{http://www.mat.univie.ac.at/~gerald/}}

\thanks{J. Spectr. Theory {\bf 5}, 663--696 (2015).}
\thanks{{\it Research supported by the Austrian Science Fund (FWF) under Grant No.\ Y330, M1329 and RFBR grants.}}


\begin{abstract}
We derive dispersion estimates for solutions of the one-dimensional discrete perturbed Schr\"odinger and wave equations.
In particular, we improve upon previous works and weaken the conditions on the potentials. To this end we also provide
new results concerning scattering for one-dimensional discrete perturbed Schr\"odinger operators which are
of independent interest. Most notably we show that the reflection and transmission coefficients belong to the
Wiener algebra.
\end{abstract}

\maketitle

\section{Introduction}

We are concerned with the one-dimensional discrete Schr\"odinger equation
\begin{equation} \label{Schr}
  \I \dot u(t)=H u(t):=(-\Delta_L  + q)\,u(t),\quad t\in\R,
\end{equation}
and the corresponding discrete wave (resp.\ Klein--Gordon) equation
\begin{equation} \label{KGE}
\ddot u(t)=(\Delta_L -\mu^2 - q)\,u(t),\quad t\in\R,\quad \mu\ge 0.
\end{equation}
with real potential $q$.
Here $\Delta_L$ is the discrete Laplacian given by
\begin{equation*}
  (\Delta_L u)_n=u_{n+1}-2u_n+u_{n-1}, \quad n\in\Z.
\end{equation*}
To formulate our results we introduce the weighted spaces $\ell^p_{\sigma}=\ell^p_{\sigma}(\Z)$,
$\sigma\in\R$, associated with the norm
\begin{equation*}
   \Vert u\Vert_{\ell^p_{\sigma}}= \begin{cases} \left( \sum_{n\in\Z} (1+|n|)^{p\sigma} |u(n)|^p\right)^{1/p}, & \quad p\in[1,\infty),\\
   \sup_{n\in\Z} (1+|n|)^{\sigma} |u(n)|, & \quad p=\infty. \end{cases}
\end{equation*}
Of course, the case $\sigma=0$ corresponds to the usual $\ell^p_0=\ell^p$ spaces without weight.

As our first main result we will prove the following $\ell^1\to \ell^\infty$ decay
\begin{equation}\label{fullp}
\Vert \E^{-\I tH}P_c\Vert_{\ell^1\to \ell^\infty}=\mathcal{O}(t^{-1/3}),\quad t\to\infty,
\end{equation}
under the assumption $q\in\ell^1_1$.
Here $P_c$ is the orthogonal projection in $\ell^2$ onto the
continuous spectrum of $H$. In this respect we recall that under the condition $q\in\ell^1_1$ it is well-known
\cite{tjac} that the spectrum of $H$ consists of a purely absolutely continuous part
covering $[0,4]$ plus a finite number of eigenvalues located in $\R\setminus[0,4]$.
In addition, there could be resonances at the edges of the continuous spectrum.

The dispersive decay \eqref{fullp} has been established by Pelinovsky and Stefanov \cite{PS}
under the assumption that there are no resonances and under the more restrictive condition
$|q_n|\le C (1+|n|)^{-\beta}$ with $\beta>5$. Cuccagna and Tarulli \cite{CT} establish
\eqref{fullp} under the assumption $q\in\ell^1_1$ if there are no resonances and under
the assumption $q\in\ell^1_2$ if there are resonances (which compares to the continuous case
established in \cite{GS}). Our main contribution here is to show that this extra decay
condition in the case of resonances is not necessary. Our novel proof is based on a simple but
useful generalization of the van der Corput lemma (Lemma~\ref{lem:vC}) together with the
novel fact that the scattering data associated with $H$ are in the Wiener algebra (Theorem~\ref{thm:scat}).
The latter result being of independent interest in scattering theory.

Moreover, \eqref{fullp} has some immediate consequences (under the same assumption $q\in\ell^1_1$).
First of all unitarity of $\exp(-\I t H) : \ell^2 \to \ell^2$ implies
\[
\Vert \E^{-\I tH}P_c(H)\Vert_{\ell^2\to \ell^2} \le 1
\]
and interpolating between this and our $\ell^1\to \ell^\infty$ estimate the Riesz--Thorin theorem gives
\begin{equation}
\Vert \E^{-\I tH}P_c(H)\Vert_{\ell^{p'}\to \ell^p} = \mathcal{O}(t^{-1/3(1/p'-1/p)})
\end{equation}
for any $p'\in[1,2]$ with $\frac{1}{p}+\frac{1}{p'}=1$. Moreover, we also can deduce some
corresponding Strichartz estimates from Theorem~1.2 of \cite{KT}. To this end we introduce the following space-time norms
\[
\|F\|_{L^q ,\ell^p} =\left( \int_\R \|F(t)\|^q_{\ell^p} d t \right)^{1/q}.
\]
Then
\begin{align}
\Vert \E^{-\I tH}P_c(H) f \Vert_{L^q, \ell^p} &\le C \|f\|_{\ell^2},\\
\Vert \int_\R \E^{-\I s H}P_c(H) F(s) ds\Vert_{\ell^2} &\le C \|F\|_{L^{q'}, \ell^{p'}},\\
\Vert \int_{s<t} \E^{-\I (t-s) H}P_c(H) F(s) ds\Vert_{L^q,\ell^p} &\le C \|F\|_{L^{q'}, \ell^{p'}},
\end{align}
where $p,q\ge 2$,
\[
\frac{1}{q}+\frac{1}{3p} \le \frac{1}{6},
\]
and a prime denotes the corresponding dual index.
Furthermore, \eqref{fullp} also implies
\[
\Vert \E^{-\I tH}P_c(H)\Vert_{\ell^2_\sigma\to \ell^2_{-\sigma}}=\mathcal{O}(t^{-1/3}),\quad t\to\infty, \quad \sigma>1/2.
\]
However, we will in fact establish the stronger result
\begin{equation}\label{as1-new}
\Vert \E^{-\I tH} P_c\Vert_{\ell^2_\sigma\to \ell^2_{-\sigma}}=\mathcal{O}(t^{-1/2}),\quad t\to\infty,\quad\sigma>1/2,
\end{equation}
which has not been obtained previously.

For the remaining results we restrict ourselves to the case when
the edges of the spectrum $\omega=0,4$ are
no resonances for the operator $H$. Then for $q\in\ell^1_2$ we show that
\begin{equation}\label{as-new}
\Vert \E^{-\I tH}P_c\Vert_{\ell^1_1\to \ell^\infty_{-1}}=\mathcal{O}(t^{-4/3}),
\quad t\to\infty.
\end{equation}
Such asymptotics with decay rate $t^{-3/2}$ were first established for continuous Schr\"o\-dinger equations 
by Schlag in \cite{S} in the case when the potential has a finite fourth moment and later refined by Goldberg \cite{G} to the case of a finite third moment.
For the discrete Schr\"odinger equations again asymptotics of the type \eqref{as-new} appear to be new.

Moreover, for $q\in\ell^1_2$ in  the non-resonant case we prove
\begin{equation}\label{as2-new}
\Vert \E^{-\I tH} P_c\Vert_{\ell^2_\sigma\to \ell^2_{-\sigma}}=\mathcal{O}(t^{-3/2}),\quad t\to\infty,\quad\sigma>3/2.
\end{equation}
Such a dispersive decay estimate was obtained for the first  time
in \cite{kkk} for discrete Schr\"odinger and Klein--Gordon equations with compactly
supported potentials. The result has been  generalized in \cite{PS} to discrete Schr\"odinger
equation with non-compactly supported potentials under the decay condition
$|q_n|\le C (1+|n|)^{-\beta}$ with $\beta>5$ and for $\sigma>5/2$.

Here we improve this result by both reducing the decay rate and the value of $\sigma$.
Again, this reduction relies on our new approach based on properties
of the Jost functions and the scattering matrix.

Finally, we  obtain similar asymptotics  for the wave (resp.\ Klein--Gordon) equation
\eqref{KGE} (except for the asymptotics in the resonant case when $\mu=0$).

In addition, we mention that asymptotics of the type \eqref{fullp}--\eqref{as2-new} play an important role in proving asymptotic
stability of solitons in the associated discrete nonlinear equations \cite{KPS,k09,PS11,SK}.
Analogous results for the continuous one-dimensional Schr\"odinger and Klein--Gordon equations will be given in \cite{EKMT}.

\section{Free discrete Schr\"odinger equation}
\label{free-sect}

As a warm-up we will first consider the free equation \eqref{Schr} with $q=0$ and denote $H_0=-\Delta_L$.
It is well-known (\cite[Sect.~1.3]{tjac}) that $H_0$ is self-adjoint and the discrete Fourier transform
\[
    \hat u(\theta)=\sum_{n\in\Z}u_n\E^{\I\theta n},
    \quad\theta\in \mathbb{T}:=\R/2\pi \Z.
\]
maps $H_0$ to the operator of multiplication by $\phi(\theta)=2-2\cos\theta$:
\[
-\widehat{H_0 u}(\theta)=\phi(\theta)\widehat u(\theta).
\]
In particular, the spectrum $\mathrm{Spec}(H_0)=[0,4]$ is purely absolutely continuous.

Adopting the notation $[K]_{n,k}$ for the kernel of an operator $K$, that is,
\[
(Ku)_n=\sum_{k\in\Z} [K]_{n,k}u_k,\quad n\in\Z,
\]
the kernel of the resolvent $\calR_0(\omega)=(H_0-\omega)^{-1}$ is given by (cf. \cite{kkk})
\begin{equation} \label{R02}
    [\calR_0(\omega)]_{n,k} =\frac 1{2\pi}\int\limits_{\mathbb{T}}\frac{\E^{-\I\theta(n-k)}}{\phi(\theta)  -\omega} d\theta
  = \frac{\E^{-\I\theta(\omega)|n-k|}}{2\I\sin\theta(\omega)},
  \quad\omega\in\Xi:=\C\setminus [0,4],
\end{equation}
$n,k\in\Z$. Here $\theta(\omega)$ is the unique solution of the equation
\begin{equation}\label{theta}
  2-2\cos\theta=\omega,\quad\theta\in \Sigma:=\{ -\pi\le\re\theta\le\pi,\;\im\theta< 0\}.
\end{equation}
Observe that $\theta\mapsto\omega=2-2\cos\theta$ is a biholomorphic map from $\Sigma\to\Xi$
with identified points $\theta=-\pi -\I a$ and $\theta=\pi - \I a$, $a\ge 0$.  Then the map $z=\E^{-\I\theta}$
is one-to-one from $\Sigma$ to the interior of the unit circle $|z|<1$. Note that the parameter $z$  is the standard spectral parameter for the Jacobi difference equation $a_{n-1}u_{n-1} +b_n u_n + a_n u_{n+1}=(z + z^{-1})u_n$, where $a_n\to 1$ and $b_n\to 0$ as $n\to\pm\infty$. The scattering theory of this equation can be found in \cite{tjac}.

The kernel of the free propagator can be easily
computed using the spectral theorem
\begin{align}\nonumber
[\E^{-\I tH_0}]_{n,k}&=\frac 1{2\pi\I}\int\limits_{[0, 4]}
\E^{-\I t\omega}[\calR_0(\omega+\I0)- \calR_0(\omega-\I0)]_{n,k}\,d\omega\\ \label{propag}
&=-\frac 1{4\pi}\int\limits_{[0, 4]} \E^{-\I t\omega}
\Big(\frac{\E^{-\I\theta_+(\omega)|n-k|}}{\sin\theta_+(\omega)}-
\frac{\E^{-\I\theta_-(\omega)|n-k|}}{\sin\theta_-(\omega)}\Big)d\omega\\ \nonumber
&=\frac 1{2\pi} \int_{-\pi}^{\pi} \E^{-\I t(2-2\cos\theta)} \E^{-\I\theta|n-k|}d\theta
\end{align}
where
\begin{equation}\label{thetadef}
\theta_+(\omega)=\theta(\omega + \I 0)\in [-\pi,0], \quad \theta_-(\omega)=\theta(\omega-\I 0)\in [0,\pi],\quad \omega\in [0,4].
\end{equation}
The last integral in \eqref{propag} is Bessel's integral implying
\begin{equation}
[\E^{-\I tH_0}]_{n,k}= \E^{\I(-2 t + \frac{\pi}{2} |n-k|)} J_{|n-k|}(2t),
\end{equation}
where $J_\nu(z)$ denotes the Bessel function of order $\nu$, \cite{W}.

For the free discrete Schr\"odinger equation the $\ell^1\to \ell^{\infty}$ decay
and  the $\ell^2_{\sigma}\to \ell^2_{-\sigma}$ decay
holds only with the rates  $t^{-1/3}$ and $t^{-1/2}$, respectively (the latter one being the same as in the  continuous case).
This is caused by the presence of resonances at the edge points $\omega=0$ and $\omega=4$.
\begin{proposition}\label{free-as}
The following asymptotics hold
\begin{equation}\label{H0-as1}
\Vert \E^{-\I tH_0}\Vert_{\ell^1\to \ell^{\infty}}=\mathcal{O}(t^{-1/3}),\quad t\to\infty,
\end{equation}
\begin{equation}\label{H0-as2}
\Vert \E^{-\I tH_0}\Vert_{\ell^2_\sigma\to \ell^2_{-\sigma}}=\mathcal{O}(t^{-1/2}),\quad t\to\infty,\quad\sigma>1/2.
\end{equation}
\end{proposition}
\begin{proof}
{\it Step i)}
To establish \eqref{H0-as1} consider  $t\ge 1$ and set $v:=|n-k|/t\ge 0$. We start from
\begin{equation}
[\E^{-\I tH_0}]_{n,k}=\frac 1{2\pi}\int\limits_{-\pi}^{\pi}\E^{-\I t(2-2\cos\theta+v\theta)}
d\theta
\end{equation}
which is an oscillatory integral with the phase function
\begin{equation}\label{phi-def}
\phi_v(\theta)=2-2\cos\theta+v\theta.
\end{equation}
The stationary points are the solution
of the equation $\phi'_v(\theta)=2\sin\theta+v=0$.
If $v>2$ the phase function has no stationary points.
For any $v<2$ the phase function
has two stationary points $\theta_{1,2}$, which
are non-degenerate, i.e.\ $\phi''_v(\theta_{1,2})\not=0$.
In the case $v=2$  the phase function has a unique degenerate stationary point
$\theta=-\pi/2$ satisfying
\begin{equation}
\phi''_2(-\pi/2)=0,\quad \phi'''_2(-\pi/2)=2\not =0.
\end{equation}
Then, since $\phi_v''(\theta)= 2\cos(\theta)$ and $\phi_v'''(\theta)= -2\sin(\theta)$, we can split our domain of
integration into four intervals where either $|\phi_v''(\theta)|\ge \sqrt{2}$ or $|\phi_v'''(\theta)|\ge \sqrt{2}$.
Applying the van der Corput lemma \cite[page 334]{St} on each interval gives \eqref{H0-as1}.
\smallskip\\
{\it Step ii)}
To establish \eqref{H0-as2} we represent  $\E^{-\I tH_0}$ as the sum
\[
\E^{-\I tH_0}=\frac 1{2\pi}(K(t)+\tilde K(t)),
\]
where
\[
[K(t)]_{n,k}=\int\limits_{|\theta+\frac{\pi}2|\le \frac{\pi}6}
\E^{-\I t\phi_v(\theta)}d\theta,\quad
[\tilde K(t)]_{n,k}= \int\limits_{|\theta+\frac{\pi}2|\ge \frac{\pi}6}\E^{-\I t\phi_v(\theta)}d\theta.
\]
By the stationary phase method we infer
\[
\sup_{n,k\in\Z}|[\tilde K(t)]_{n,k}|\le C t^{-1/2},\quad t\ge 1,
\]
implying \eqref{H0-as2} for $\tilde K(t)$. The required estimate for $K(t)$ will follow from the next lemma.
\begin{lemma}\label{Lem:K}
For any $\sigma>1/2$ the following estimate holds:
\begin{equation}\label{K-est}
\sum\limits_{n,k\in\Z}[K(t)]^2_{n,k}
\frac{1}{(1+|n|)^{2\sigma}(1+|k|)^{2\sigma}}\le Ct^{-1}.
\end{equation}
\end{lemma}
\begin{proof}
For any fixed  $\sigma>1/2$,  there exist an integer $N>0$ such that
\begin{equation}\label{sigma}
\sigma>1/2+(1/2)^{N}.
\end{equation}
Denote $t_j=t^{-(\frac 12)^j}$, $1\le j\le N$,  $t_0=0$, $t_{N+1}=\pi/6$, and  represent $K(t)$ as the sum
\[
K(t)=\sum_{j=0}^{N}K_{j}(t),
\]
where  $K_{j}(t)$, $0\le j\le N$,
is the integral over $t_j\le |\theta+\frac{\pi}2|\le t_{j+1}$.
We will establish a bound of type \eqref{K-est} for each summand.
For $K_{0}(t)$ the bound evidently holds.
Furthermore, by the van der Corput Lemma
\begin{equation}\label{C1}
\sup_{n,k\in\Z}|\!\!\!\int\limits_{t_j\le|\theta+\frac{\pi}2|\le a}
\!\!\!\!\!\E ^{-\I t\phi_v(\theta)}d\theta|
\le Ct^{-1/2}\left(\min_{t_j\le|\theta+\frac{\pi}2|\le \frac{\pi}6}|\phi''_v(\theta)|\right)^{-1/2}
\!\!\!\le C(tt_j)^{-1/2}
\end{equation}
for any $a\in[t_j,\pi/6]$ implying
\begin{equation}\label{C7}
\sup_{n,k\in\Z}|[K_j(t)]_{n,k}|\le Ct^{-1/2}t_j^{-1/2},\quad j=1,\dots,N.
\end{equation}
To get the estimate \eqref{K-est} for each $K_j(t)$, $1\le j\le N$,
we choose  $\ve=2^{-N}$, so that $t^\ve=t_N^{-1}$, and consider
two different cases: $|2-v|\in [0,t_jt^{\ve}]$ and
$|2-v|\in[t_jt^{\ve},2]$ separately.

In the first case we take $T_j=\left\{(n,k)\in\Z^2:\ |2t-|n-k||\le t_jt^{1+\ve}\right\}$
as the domain of summation. Since this domain is symmetric  with respect to the
map $(n,k)\mapsto (-n, -k)$, we can make the change of variables $p=n-k$, $q=n+k$ and estimate
\[
b_j(t):=\sum_{(n,k)\in T_j}\frac {1}{(1+|n|)^{2\sigma}(1+|k|)^{2\sigma}}
\]
as
\[
b_j(t)\le \sum_{q\in\Z}\sum_{p=\ceil{2t-t_jt^{1+\ve}}}^{\floor{2t+t_jt^{1+\ve}}}
\frac {2}{(1+\frac{1}{2}|p+q|)^{2\sigma}(1+\frac{1}{2}|p-q|)^{2\sigma}},
\]
where $\floor{\cdot}$, $\ceil{\cdot}$ denote the usual floor and ceiling functions.
The sum with respect to $p$  is finite with the number of summands less then
$2\floor{t_jt^{1+\ve}}+2$. Since $t_jt^{1+\ve}\le t$ for $j=1,\dots, N$
 we have $p\geq t$ in the domain of summation.
Consequently $p+q\geq t$ for $q\geq 0$ and $p-q\geq t$ for $q<0$.
Using these estimates and interchanging the order of summation we get
\begin{equation}\label{GGG}
 b_j(t)\leq C\,\frac{\floor{t_jt^{1+\ve}}}{t^{2\sigma}}\leq C t_jt^{1+\ve-2\sigma}.
\end{equation}
Thus, by \eqref{C7}, \eqref{GGG}, and \eqref{sigma}
\begin{equation}\label{GG}
\sum_{n,k\in T_j}\frac {([K_j(t)]_{n,k})^2}{(1+|n|)^{2\sigma}(1+|k|)^{2\sigma}}
\leq\sup_{n,k\in \Z}([K_j(t)]_{n,k})^2 b_j(t)\leq C t^{-2\sigma + \ve}\leq C t^{-1}.
\end{equation}
In the second case $(n,k)\notin T_j$ we have (using $\theta+\pi/2=\psi$)
\begin{align*}
|[K_j(t)]_{n,k}|&=\left|\;\int\limits_{t_j\le|\psi|\le t_{j+1}} \!\!\!\! \E^{-\I t(v\psi-2\sin\psi)}d\psi \right|
=2\left|\;\int\limits_{t_j\le\psi\le t_{j+1}} \!\!\!\! \cos\big( t(v\psi-2\sin\psi) \big)d\psi \right|.
\end{align*}
Applying integration by parts we get
\begin{align}\label{denom}
|[K_j(t)]_{n,k}|
&\le \frac{2}{t}\Big(\frac{1}{|v-2\cos t_j|} +\frac{1}{|v-2\cos t_{j+1}|}
+\int_{t_j}^{t_{j+1}} \!\! \frac{2\sin(\psi) d\psi}{(v-2\cos\psi)^2}\Big)\\
\nonumber
&\le \frac{4}{t}\left(\frac{1}{|4\sin^2(t_j/2)+v-2|}+
\frac{1}{|4\sin^2(t_{j+1}/2)+v-2|}\right).
\end{align}
Since for $j=1,\dots,N-1$ we have $|v-2|\ge t_jt^{\ve}\gg t_{j+1}^2=t_{j}> t_j^2$, we see
\begin{equation}\label{dop15}
|4\sin^2\frac{t_{j+s}}{2}+v-2|\geq |v-2| - 4\sin^2\frac{t_{j+s}}{2}\geq |v-2| - t_{j+s}^2\geq |v-2|-t_j
\end{equation}
for $s=0,1$.
But $|v-2| - t_j\geq t_j (t^{\ve}-1)\geq C t_j$, therefore
\begin{equation}\label{dop14}
\sup\limits_{(n,k)\notin T_j}|[K_j(t)]_{n,k}|\le Ct^{-1}t_j^{-1}\le Ct^{-1/2},\qquad j=1,\dots,N-1.
\end{equation}
For $j=N$ we have $|v-2|\geq 1$ and thus $4\sin^2(t_{N+1}/2)=4\sin^2(\pi/12)< |v-2|/2$.  Respectively,
$|4\sin^2(t_{N+1}/2)+v-2|\geq |v-2|/2,$ which implies $|[K_N(t)]_{n,k}|\le Ct^{-1}$.
Combining this with \eqref{GG} we get \eqref{K-est} for each $ K_j(t)$ as $1\le j\le N$.
\end{proof}
\noindent
This finishes the proof of Proposition \ref{free-as}.
\end{proof}
\begin{remark}
The decay rate in \eqref{H0-as1}  is ``sharp'' as can be seen from the following asymptotics of the Bessel function
\[
 J_t(t)\sim t^{-1/3},\quad t\to \infty,
\]
see \cite[Section 8.2]{W}.
\end{remark}
\section{Jost solutions and the resolvent}
\label{st-sec}
Consider the  Jost solutions
$f^\pm(\theta)$ to the equation
\begin{equation*}
Hf:=(-\Delta_L+q)f=\omega f,
\end{equation*}
normalized as
\begin{equation*}
f^\pm_n(\theta)\sim  \E^{\mp \I n \theta},\quad n \to \pm \infty,
\end{equation*}
where $\omega\in {\overline\Xi}$ and $\theta=\theta(\omega)\in{\overline\Sigma}$  (cf.\ \eqref{theta}).
For $q\in\ell_1^1$ this solution exists everywhere in $\overline\Xi$, but for $q\in\ell^1$ it exists outside of the edges of continuous spectrum.
Introduce
\begin{equation}\label{Jostcut}
h^\pm_n(\theta)=\E^{\pm \I n \theta}f^\pm_n(\theta)
\end{equation}
and set
\[
\overline\Sigma_\delta:=\{\theta\in\overline\Sigma:\, |\E^{-\I\theta} \pm 1|>\delta
\},\quad 0<\delta<\sqrt 2.
\]

\begin{lemma}\label{newl}
(i) Let $q\in\ell^1_s$ with $s=0,1,2$. Then the functions $h^\pm_n(\theta)$ can be differentiated $s$ times on $\overline\Sigma_\delta$, and  the following estimates hold:
\begin{equation}\label{dh-est}
|\frac{\partial^p}{\partial \theta^p} h^\pm_n(\theta) |\le C(\delta)\max ((\mp n)|n|^{p-1}, 1), \quad n\in\Z,
\quad 0\le p\le s,\quad \theta\in \overline\Sigma_\delta.
\end{equation}\\
(ii) If additionally  $q\in \ell_{s+1}^1$, then $h^\pm_n(\theta)$ can be differentiated $s$ times on
$\overline\Sigma$,
and the following estimates hold:
\begin{equation}\label{Jostderiv}
|\frac{\partial^{p}}{\partial \theta^{p}} h^\pm_n(\theta) |\le C\max ((\mp n)|n|^{p}, 1),\quad n\in\Z,
\quad 0\le p\le s,\quad\theta\in\overline\Sigma.
\end{equation}
\end{lemma}
\begin{proof}
The proof of \eqref{dh-est} is similar for $``+"$ and $``-"$ cases, hence we give it only for the $``+"$ case. Denote
 $h_n(z)=h_n^+(\theta)$ with $z=\E^{-\I\theta}$, $|z|\leq 1$.
Function $h_n(z)$ satisfies the integral equation (see \cite{tjac})
\begin{equation}\label{green}
h_n(z)=1 +\sum_{m=n+1}^\infty G(n,m,z) h_m(z),\quad  G(n,m,z):=q_m\frac{z^{2m - 2n} - 1}{z^{-1} - z}.
\end{equation}
For $\theta\in \overline\Sigma_\delta$ we have  $|z^2 -1|\geq C(\delta)>0$. Then
\[
|G(n,m,z)|\leq \frac{2|z| |q_m|}{|z^2 - 1|}\leq C(\delta) |q_m|,\quad m-n> 0,
\]
and the method of successive approximations as in \cite{EG} implies $|h_n(z)|\leq C(\delta)$.
Then \eqref{dh-est} with $p=0$ follows.
Further,
\begin{equation}
\label{ddots}|\frac{d^p}{dz^p} G(n,m,z)|\leq C(\delta) (m-n)^p |q_m|,\quad p\ge 1,\quad m-n> 0,
\quad \theta\in \overline\Sigma_\delta.
\end{equation}
Now let $q\in\ell^1_1$. Consider the first derivative of $h_n(z)$. We have
\begin{equation}\label{appr}
\frac{d}{dz} h_n(z)= \phi_n(z)+\sum_{m=n+1}^\infty G(n,m,z) \frac{d}{dz}h_m(z),
\end{equation}
where
\[
\phi_n(z):=\sum_{m=n+1}^\infty h_m(z)\frac{d}{dz} G(n,m,z)
\]
with $|\phi_n(z)|\leq C(\delta)$ as $ n\geq 0$ and $\theta\in \overline\Sigma_\delta$
by \eqref{dh-est} with $p=0$ and \eqref{ddots}.
Applying the method of successive approximations  to \eqref{appr} we get \eqref{dh-est} with $p=1$.
For the case $p= 2$ we proceed in the same way.

The estimate \eqref{Jostderiv} can be obtained from \eqref{green} by the same approach by virtue of the estimate
$|\frac{d^p}{dz^p}G(n,m,z)|\leq 2|q_m|(m-n)^{p+1}$, which is valid for all $|z|\leq 1$ and $m>n$.
\end{proof}
\begin{corollary} In the case $q\in \ell^1$, Lemma \ref{newl} (i)  implies in particular that for any
$\theta\in\overline\Sigma\setminus\{0,\pi,-\pi\}$
we got the estimate $|h_n^\pm(\theta)|\leq C(\theta)$ for all $n\in\Z$, where $C(\theta)$ can be chosen uniformly in compact subsets of $\overline{\Sigma}$ avoiding the band edges. Together with \eqref{Jostcut} this implies
\begin{equation}\label{De0}
| f^{\pm}_n(\theta) |\le C(\theta) \E^{\pm \im(\theta)n}\, ,\quad \theta\in\overline\Sigma\setminus\{0,\pi,-\pi\},\quad n\in\Z.
\end{equation}
\end{corollary}

Given the Jost solutions we can express the kernel of the resolvent $\calR(\omega)=(H-\omega)^{-1}:\ell^2 \to \ell^2$ for
$\omega\in \C\setminus\mathrm{spec}(H)$ as (cf.\ \cite[(1.99)]{tjac})
\begin{equation}\label{RJ1-rep}
[\calR(\omega)]_{n,k} = \frac{1}{W(\theta(\omega))} \left\{ \begin{array}{cc}
f_n^+(\theta(\omega)) f_k^-(\theta(\omega))
\;\; \mbox{for} \;\; n \ge k, \\[2mm]
f_k^+(\theta(\omega)) f_n^-(\theta(\omega))
 \;\; \mbox{for} \;\; n\le k, \end{array} \right.
\end{equation}
where
\begin{equation}\label{alg2}
W(\theta): = W(f^+(\theta), f^-(\theta))=f^+_0(\theta) f^-_{1}(\theta) - f^+_{1}(\theta) f^-_0(\theta)
\end{equation}
is the Wronskian of the Jost solutions.
Recall that $\theta\mapsto\omega(\theta)$ is a biholomorphic map $\Sigma\to\Xi$.

The representation \eqref{RJ1-rep}, the fact that $W(\theta)$ does not vanish for $\omega\in(0,4)$,
and the bound \eqref{De0} imply the
limiting absorption principle for the perturbed one-dimensional Schr\"odinger equation.
\begin{lemma}\label{BV}
Let $q\in\ell^1$. Then the convergence
\begin{equation}\label{esk}
   \calR(\omega\pm \I\varepsilon)\to \calR(\omega\pm \I0),\quad\varepsilon\to 0+,\quad \omega\in (0,4)
\end{equation}
holds in ${\mathcal L}(\ell^2_\sigma,\ell^2_{-\sigma})$ with $\sigma>1/2$.
\end{lemma}
\begin{proof}
For any $\omega\in (0,4)$ and any $n,k\in\Z$, there exist the pointwise limit
\[
[\calR(\omega\pm \I\ve)]_{n,k}\to[\calR(\omega\pm \I 0)]_{n,k},\quad \ve\to 0.
\]
Moreover, the bound \eqref{De0} implies that $|[\calR(\omega\pm \I\ve)]_{n,k}|\le C(\omega)$.
Hence, the Hilbert--Schmidt norm of the difference
$\calR(\omega\pm \I\ve)-\calR(\omega\pm \I 0)$
converges to zero in $B(\sigma,-\sigma)$ with $\sigma>1/2$ by the Lebesgue dominated convergence theorem.
\end{proof}
\begin{corollary}\label{J-rep}
For any $\omega\in (0,4)$ and any fixed $\sigma>1/2$, the operators
$\calR(\omega\pm i0):\ell^2_\sigma\to \ell^2_{-\sigma}$
have integral kernels given by
\begin{equation}\label{RJ-rep}
[\calR(\omega\pm i0)]_{n,k} = \frac{1}{W(\theta_\pm)}\left\{ \begin{array}{cc}
f_n^+(\theta_\pm) f_k^-(\theta_\pm) \;\; \mbox{for} \;\; n \ge k \\\\
f_k^+(\theta_\pm) f_n^-(\theta_\pm) \;\; \mbox{for} \;\; n\le k \end{array} \right.
\end{equation}
where $\theta_+$, and $\theta_-=-\theta_+$
are defined by \eqref{thetadef}.
\end{corollary}

At the end of this section we discuss an alternative definition of resonances.
\begin{definition}
For $\omega\in\{0,4\}$
any nonzero solution $u\in\ell^{\infty}(\Z)$  of the equation  $H u=\omega u$
is called a resonance function,
and in this case the point $\omega$  is called a resonance.
\end{definition}

\begin{lemma}\label{W0}
Let $q\in\ell^1_1$. Then $\omega=0$ (or~$\omega=4$) is a resonance
if and only if $W(0)=0$ (or $W(\pi)=0$).
\end{lemma}
\begin{proof}
We  consider the case $\omega=0$. In this case
$
f^\pm_n= 1+o(1),$ as $n\to\pm\infty.
$
Introduce another solution $g^+$  satisfying $W(f^+,g^+)=1$. Making the ansatz $g^+_n~=~f^+_n~v_n$,
where $v_n$ is unknown, we obtain $(v_{n+1} - v_n)f_n^+f_{n+1}^+=1$ for sufficiently large positive $n_0$.
Solving for $v$ shows
\[
g^+_n= f^+_n \sum_{j=n_0}^{n-1} \frac{1}{f^+_j f^+_{j+1}} + v_{n_0} f_n^+= n+o(n),\quad n\to+\infty.
\]
Hence $f^-_n = \alpha f^+_n + \beta g^+_n$ and there is a bounded solution if and only if $\beta = W(f^+,f^-) =0$.
\end{proof}
\section{Properties of the scattering matrix}
\label{wa-sec}
Recall that the Wiener algebra is the set of all integrable functions
whose Fourier coefficients are integrable:
\[
\mathcal{A} = \Big\{ f(\theta) = \sum_{m\in\Z} \hat{f}_m \E^{\I m \theta}\  \Big|\, \|\hat{f}\|_{\ell^1} < \infty \Big\}.
\]
We set
\begin{equation}\label{n-A}
\Vert f\Vert_{\mathcal{A}}=\Vert \hat f\Vert_{\ell^1}.
\end{equation}
In the  case $q\in\ell_1^1$ the functions $h^\pm_n$ from \eqref{Jostcut} can be represented as
\begin{equation}\label{fh}
 h^\pm_n(\theta) =
1 + \sum_{m=\pm 1}^{\pm\infty} B^\pm_{n,m} \E^{\mp\I m \theta},
\end{equation}
where (see \cite[Sect.~10.1]{tjac}) $B^\pm_{n,m}\in\R$ and
\begin{equation}\label{est3}
|B^\pm_{n,m}| \le  C^\pm_n  \sum_{k=n+\lfloor m/2 \rfloor}^{\pm\infty} |q_k|,
\end{equation}
with
\begin{equation}\label{estC}
\quad C^\pm_n\le C^\pm\quad \mbox{if}\  \pm n\geq \mp 1.
\end{equation}
The estimate \eqref{est3} implies
\begin{equation} \label{alg1}
h^\pm_n(\theta), f^\pm_n(\theta)\in \mathcal A \quad \mbox{if} \quad q\in\ell^1_1.
\end{equation}
Moreover, the Wronskian $W(\theta)$ (see \eqref{alg2}) of Jost solutions
also belongs to the Wiener algebra $\mathcal A$ if $q\in\ell^1_1$ and the same holds true for the Wronskians
$W^\pm(\theta)=W(f^\mp(\theta), f^\pm(-\theta))$.
Moreover, we have the scattering relations
\begin{equation}\label{scat-rel}
T(\theta)f^\pm_m(\theta)=R^\mp(\theta)f_m^\mp(\theta) + f_m^\mp(-\theta),\quad \theta\in [-\pi, \pi],
\end{equation}
where the quantities
\begin{equation}
T(\theta)= \frac{2\I\sin\theta}{W(\theta)},\quad R^\pm(\theta)= \pm\frac{W^\pm(\theta)}{W(\theta)},
\end{equation}
which are known as the transmission and reflection coefficients, also belong to this algebra:

\begin{theorem}\label{thm:scat}
If $q\in\ell^1_1$, then  $T(\theta)$, $R^\pm(\theta)\in \mathcal{A}$.
\end{theorem}

\begin{proof} The Wronskian $W(\theta)$ can vanish only at the edges of continuous spectra,
i.e. when $\theta=0,\pm\pi$, which correspond to the resonant cases (see Lemma \ref{W0} below).
Remind that we identify points $\pi$ and $-\pi$, considering Jost solutions,
Wronskians and scattering data as functions on the unit circle.
Thus it is sufficient to consider the points $0$ and $\pi$.
Since $|T(\theta)|\leq 1$  as $\theta\in[-\pi, \pi]$ then the zeros of the Wronskian at points $0,\pi$
can be at most of first order.
Since $W(\theta) \in \mathcal{A}$ by \eqref{alg1}, then in the  case $W(0)W(\pi)\neq 0$
we obtain $W(\theta)^{-1} \in \mathcal{A}$ by Wiener's lemma. Therefore, $T, R^\pm \in \mathcal A$.

If $W(0)W(\pi)=0$ we need to work a bit harder. Suppose, for example, $W(0)=0$. In \cite{emt},
Lemma 4.1, formulas (4.12)--(4.14), the following representation is obtained
\begin{equation}\label{alg6}
 V^\pm(\theta):=f_1^\pm(\theta)f_0^\pm(0) - f^\pm_0(\theta)f^\pm_1(0)=(1 - \E^{\I\theta})\Psi^\pm(\theta),
\end{equation}
where
\begin{equation}\label{alg8}
\Psi^\pm(\theta)=\sum_{l=\frac{1\pm 1}{2}}^{\pm \infty} g_m^\pm\E^{\mp\I m\theta},
\quad \mbox{with}\quad g^\pm\in\ell^1(\mathbb Z_\pm)\quad \mbox{if}\ q\in\ell_1^1.
\end{equation}
In other words, $\Psi^\pm(\theta)\in\mathcal A$.
Since
\begin{equation}\label{alg4}
W(0)= f^+_0(0) f^-_{1}(0) - f^+_{1}(0) f^-_0(0)=0
\end{equation}
we have two possible combinations (since the solutions $f^\pm_m(0)$ cannot vanish at two
consecutive points): (a)\
$f^+_0(0)f^-_0(0)\neq 0$ and (b)\ $f^+_1(0)f^-_1(0)\neq 0$. Consider the case (a).
By \eqref{alg2}, \eqref{alg6}, and \eqref{alg4} we get
\begin{align*}
W(\theta) &=f_0^+(\theta)f_0^-(\theta)\left(\frac{V^-(\theta)}{f_0^-(0)f_0^-(\theta)}-\frac{V^+(\theta)}{f_0^+(0)f_0^+(\theta)}
\right)=\\
&=(1-\E^{\I\theta})\left( \frac{f^+_0(\theta)}{f^-_0(0)}
\Psi^-(\theta)-\frac{f^-_0(\theta)}{f^+_0(0)}\Psi^+(\theta)\right)=(1-\E^{\I\theta})\Phi(\theta),
\end{align*}
where $\Phi(\theta)\in\mathcal A$ by \eqref{alg8} and \eqref{alg1}. We observe that if $W(\pi)=0$ then
$\Phi(\theta)\neq 0$ for $\theta\in (-\pi,\pi)$  and if
$W(\pi)\neq 0$ then $\Phi(\theta)\neq 0$ for $\theta\in[-\pi,\pi]$.
The same result follows in a similar fashion in case (b).
Since equality $W(0)=0$ implies $W^\pm(0)=0$ then  we can also get similarly
$W^\pm(\theta)=(1 -\E^{\I\theta})\Phi^\pm(\theta)$ with $\Phi^\pm(\theta)\in\mathcal A$.

Analogously, $W(\pi)=0$ implies $W(\theta)= (1+\E^{\I\theta}) \tilde{\Phi}(\theta)$,
$W^\pm(\theta)= (1+\E^{\I\theta}) \tilde{\Phi}^\pm(\theta)$ with
$\tilde{\Phi},\tilde{\Phi}^\pm\in \mathcal A$ and $\tilde\Phi(\theta)\neq 0$ for $\theta\in[-\pi,\pi]$ if $W(0)\neq 0$.
Thus if $W$ vanishes at only one edge of spectrum, this finishes the proof. If $W$ vanishes at both edges, then
we can use a smooth cut-off function to combine both representations into
$W(\theta)= (1-\E^{2\I\theta}) \breve{\Phi}(\theta)$ (respectively,
$W^\pm(\theta)= (1-\E^{2\I\theta}) \breve{\Phi}^\pm(\theta)$) with $\breve{\Phi},\breve{\Phi}^\pm\in \mathcal A$
and $\breve{\Phi}(\theta)\neq 0$ for $\theta\in[-\pi,\pi]$.
\end{proof}
\section{Dispersive decay in the resonant case}
\label{ll-sec}
We begin with a small variant of the van der Corput lemma.
\begin{lemma}\label{lem:vC}
Consider the oscillatory integral
\begin{equation}
I(t) = \int_a^b \E^{\I t \phi(\theta)} f(\theta) d\theta, \qquad -\pi \le a < b \le \pi,
\end{equation}
where $\phi(\theta)$ is real-valued.
If $\min\limits_{\theta\in[a,b]}|\phi^{(s)}(\theta)|=m_s>0$ for some $s\ge 2$ and $f\in\mathcal{A}$, then
\begin{equation}
|I(t)| \le \frac{C_s \|\hat{f}\|_{\ell^1}}{(m_s t)^{1/s}}, \quad t\ge 1,
\end{equation}
where $C_s$ is a universal constant.
\end{lemma}

\begin{proof}
We rewrite
\[
I(t) =  \int_a^b \E^{\I t \phi(\theta)} \sum_{p\in\Z} \hat{f}_p\E^{\I p \theta} d\theta = \sum_{p\in\Z} \hat{f}_p I_{p/t}(t), \quad I_v(t)= \int_a^b \E^{\I t (\phi(\theta) + v \theta)} d\theta.
\]
By the van der Corput lemma \cite[page 332]{St} we have $|I_v(t)| \le C_s (m_s t)^{-1/s}$, where $C_s$ is a universal constant (independent of $v$) and the claim follows.
\end{proof}

\begin{remark}
The above lemma is usually found for the case when $f$ is absolutely continuous in the literature (cf.\ \cite[page 333]{St}) --- in fact, the proof immediately extends to
functions of bounded variation. However, by the Riemann--Lebesgue lemma the Fourier coefficients of an absolutely continuous function must satisfy $\hat{f}_m = o(m^{-1})$
(for functions of bounded variation one has $O(m^{-1})$) and considering lacunary Fourier coefficients one obtains
an element in the Wiener algebra which is not absolutely continuous (of bounded variation). Conversely, since the Fourier coefficients of an integrable function
can have arbitrary slow decay, there are absolutely continuous functions which are not in the Wiener algebra.
Finally, note that for continuous $f$ the decay can be arbitrary slow.
\end{remark}

Now we come to our main result in this section.

\begin {theorem}\label{end1}
 Let $q\in\ell^1_1$. Then the asymptotics \eqref{fullp} and \eqref{as1-new} hold, i.e.,
\begin{equation}\label{Schr-as1}
\Vert \E^{-\I tH}P_c\Vert_{\ell^1\to \ell^\infty}=\mathcal{O}(t^{-1/3}),\quad t\to\infty,
\end{equation}
\begin{equation}\label{Schr-as2}
\Vert \E^{-\I tH} P_c\Vert_{\ell^2_\sigma\to \ell^2_{-\sigma}}=\mathcal{O}(t^{-1/2}),\quad t\to\infty,\quad\sigma>1/2.
\end{equation}
\end{theorem}
\begin{proof}
{\it Step i)}
We apply the spectral representation
\begin{equation}\label{spR}
   \E^{-\I tH}P_c
   =\frac 1{2\pi \I}\int\limits_{[0, 4]}
   \E^{-\I t\omega}( \calR(\omega+\I 0)- \calR(\omega-\I 0))\,d\omega.
\end{equation}
Expressing the kernel of the resolvent in terms of the Jost solutions (cf.\ \cite[(1.99)]{tjac}),
the kernel of $\E^{-\I t H} P_c$ reads:
\begin{align}\label{HP1}
\left[ \E^{-\I t H}P_c \right]_{n,k} & =  \frac{1}{2 \pi \I} \int_0^4 \E^{-\I t \omega}
\left[ \frac{f_k^+(\theta_+) f_n^-(\theta_+)}{W(\theta_+)}
- \frac{f_k^+(\theta_-) f_n^-(\theta_-)}{W(\theta_-)} \right] d\omega \\ \nonumber
& =  -\frac{1}{\pi \I} \int_{-\pi}^{\pi} \E^{-\I t (2 - 2\cos \theta)}
\frac{f_k^+(\theta) f_n^-(\theta)}{W(\theta)}  \sin\theta\, d\theta
\end{align}
for $n\le k$ and by symmetry $\left[ \E^{-\I t H}P_c \right]_{n,k}= \left[ \E^{-\I t H}P_c \right]_{k,n}$ for $n\ge k$.
Hence, for \eqref{Schr-as1} it suffices to prove that
\begin{equation}
\label{eq:m21}
\left[ \E^{-\I tH}P_c\right]_{n,k} =\mathcal{O}( t^{-1/3}),\quad t\to\infty.
\end{equation}
independent of $n,k$. We suppose $n\le k$ for notational simplicity.
Then
\[
\left[ \E^{-\I t H}P_c \right]_{n,k} =  \frac{1}{2\pi} \int_{-\pi}^{\pi}
\E^{-\I t\phi_v(\theta)} h_k^+(\theta)h_n^-(\theta) T(\theta) d\theta,
\]
where $\phi_v$ is defined in \eqref{phi-def} with $v=\frac{k-n}{t}\geq 0$. We observe that the function
\begin{equation}\label{igrek}Y_{n,k}(\theta)=h_k^+(\theta)h_n^-(\theta) T(\theta)\end{equation} belongs to $\mathcal A$, moreover,
the $\ell^1$-norm of its Fourier coefficients $\hat Y_{n,k}(\cdot)$
can be estimated by a value, which does not depend on $n$ and $k$.
To this end introduce
\[
1+\sup_{\pm n >0}\sum_{m=1}^{\pm\infty }|B_{n,m}^\pm|=\tilde{C}^\pm>0.
\]
By \eqref{est3}--\eqref{estC} this supremum is finite. Then
\begin{equation}\label{est10}
\|\hat h^\pm_n(\cdot)\|_{\ell^1}\leq \tilde{C}^\pm\quad\mbox{for}\quad\pm n>0.
\end{equation}
Now consider the three possibilities (a) $n\leq k\leq 0$, (b) $0\leq n\leq k$ and (c) $n\leq 0\leq k$.
In the case (c) the bound \eqref{est10} and Theorem~\ref{thm:scat} imply
\begin{equation}\label{Phi-est}
\|\hat Y_{n,k}(\cdot)\|_{\ell^1} \leq C.
\end{equation}
In the other two cases  we use the scattering relations \eqref{scat-rel} to get the representation
\begin{equation}\label{Phi_nk}
Y_{n,k}(\theta)=\left\{\begin{array}{ll}
 h_n^-(\theta)(R^-(\theta)h_k^-(\theta)\E^{2\I k\theta} + h_k^-(-\theta)) & n\leq k\leq 0,\\[2mm]
 h_k^+(\theta)(R^+(\theta)h_n^+(\theta)\E^{-2\I n\theta} + h_n^+(-\theta)) & 0\leq n\leq k,
\end{array}\right.
\end{equation}
and again apply Theorem~\ref{thm:scat} together with \eqref{alg1} and \eqref{est10} to obtain \eqref{Phi-est}.

Now, as in the proof of \eqref{H0-as1} (see step (i) in the proof of Proposition~\ref{free-as})
we split the domain of integration into regions where
either the second or third derivative of the phase is nonzero and apply Lemma~\ref{lem:vC} together with the
estimates from Theorem~\ref{thm:scat}.
\smallskip\\
{\it Step ii)}
Set $J:=\{\theta\in[-\pi, \pi]: \left|\theta\pm\frac{\pi}{2}\right|\leq\frac{\pi}{6}\}$.
To establish \eqref{Schr-as2}
we represent $\left[ \E^{-\I t H}P_c \right]_{n,k}$ as the sum
\[
\left[ \E^{-\I t H}P_c \right]_{n,k}=[{\mathcal K}^\pm(t)]_{n,k}+[\tilde{\mathcal K}(t)]_{n,k},
\]
where
\begin{align*}
[{\mathcal K}^\pm(t)]_{n,k}&=\frac{1}{2\pi} \int_{|\theta\pm\frac{\pi}2|\le\frac{\pi}6}
\E^{-\I t\phi_v(\theta)} Y_{n,k}(\theta) d\theta,\\
[\tilde{\mathcal K}(t)]_{n,k}&=\frac{1}{2\pi} \int_{\theta\in[-\pi,\pi]\setminus J}
\E^{-\I t\phi_v(\theta)} Y_{n,k}(\theta) d\theta,
\end{align*}
and $Y_{n,k}(\theta)=h_k^+(\theta)h_n^-(\theta) T(\theta)$ as above.
Lemma \ref{lem:vC} with $s=2$ and the bound \eqref{Phi-est} imply
\[
\sup_{n,k\in\Z}|[\tilde{\mathcal K}(t)]_{n,k}|\le C t^{-1/2},\quad t\ge 1.
\]
Then
\[
\Vert \tilde{\mathcal K}(t)\Vert_{\ell^2_\sigma\to\ell^2_{-\sigma}}\le C t^{-1/2},\quad\sigma>1/2,
\quad t\ge 1.
\]
It remains to obtain the same estimate for ${\mathcal K}^\pm(t)$.
Since $W(\theta)\neq 0$ for $\theta\in J$, it follows from Lemma \ref{newl} that
\begin{equation}\label{dTR-est}
|\frac{d}{d\theta}T(\theta)|,\quad |\frac{d}{d\theta}R^\pm(\theta)|\le C,\quad \theta\in J.
\end{equation}
Furthermore, we split ${\mathcal K}^\pm(t)$ as
\[
{\mathcal K}^\pm(t)={\mathcal K}^\pm_a(t)+{\mathcal K}^\pm_b(t)+{\mathcal K}^\pm_c(t),
\]
where ${\mathcal K}^\pm_a(t)$ are
the restrictions of the operators ${\mathcal K}^\pm(t)$ to the case (a) $n\le k\le 0$ etc.
First we estimate ${\mathcal K}^\pm_c(t)$.
The bounds \eqref{dh-est} and \eqref{dTR-est} imply
\begin{equation}\label{phi-c}
|\frac{\partial}{\partial \theta}Y_{n,k}(\theta)|\le C,~~ \theta\in J,~~n\le 0\le k.
\end{equation}
Therefore, applying integration by parts, we obtain
\[
|[{\mathcal K}^-_c(t)]_{n,k}|\le C t^{-1},\quad t\ge 1,
\]
and then
\[
\Vert {\mathcal K}^-_c(t)\Vert_{\ell^2_\sigma\to\ell^2_{-\sigma}}\le C t^{-1},\quad\sigma>1/2,
\quad t\ge 1.
\]
To estimate ${\mathcal K}^+_c(t)$
we apply the general scheme of Lemma \ref{Lem:K}. In particular, to prove a bound of the type \eqref{C1}
for the integral with the additional factor $Y_{n,k}$ we use
Lemma \ref{lem:vC}. To get \eqref{denom} we use the bounds \eqref{Phi-est} and \eqref{phi-c}.
For the other estimates we repeat literally the respective estimates of Lemma \ref{Lem:K}. Thus, we obtain
\[
\Vert {\mathcal K}^+_c(t)\Vert_{\ell^2_\sigma\to\ell^2_{-\sigma}}\le C t^{-1/2},\quad\sigma>1/2,
\quad t\to\infty.
\]
Now consider the case (a). Using the first line of \eqref{Phi_nk} and the fact
\begin{equation}\label{nk}
\I(n-k)+2\I k=\I(k+n)=-\I|k+n|,\quad n\le k\le 0,
\end{equation}
we represent ${\mathcal K}^\pm_a(t)$  as
\[
{\mathcal K}^\pm_a(t)=\frac{1}{2\pi} \int_{|\theta\pm\frac{\pi}2|\le\frac{\pi}6}
\E^{-\I t\phi_v(\theta)} Y^1_{n,k}(\theta) d\theta
+\frac{1}{2\pi} \int_{|\theta\pm\frac{\pi}2|\le\frac{\pi}6}
\E^{-\I t{\tilde\phi}_v(\theta)} Y^2_{n,k}(\theta) d\theta
\]
where $Y_{n,k}^1(\theta)=h_n^-(\theta)h_k^-(-\theta)$,
$Y_{n,k}^2(\theta)=R^-(\theta)h_n^-(\theta)h_k^-(\theta)$, and
\begin{equation}\label{tphi}
\tilde\phi_v(\theta)=2-2\cos\theta+\tilde v\theta,~~{\rm with}~~ \tilde v=|n+k|/t\ge 0.
\end{equation}
Since $$
|\frac{\partial}{\partial \theta}Y_{n,k}^j(\theta)|\le C,~~ \theta\in J,~~n\le k\le 0,~~j=1,2,
$$
then ${\mathcal K}^\pm_a(t)$
can be treated similarly to ${\mathcal K}^\pm_c(t)$.

In the case (b) we have
\begin{equation}\label{nk1}
\I(n-k)-2\I n=-\I(k+n)=-\I|k+n|,\quad 0\le n\le k.
\end{equation}
and then the proof is the same as in the case (a).
\end{proof}
\section{Dispersive decay in the non-resonant case}
\label{ll2-sec}
\begin {theorem}\label{t-new}
Let $q\in\ell^1_2$. Then in the non-resonant case the asymptotics \eqref{as-new} hold, i.e.,
\begin{equation}\label{Schr-asn1}
\Vert \E^{-\I tH}P_c\Vert_{\ell^1_1\to \ell^\infty_{-1}}=\mathcal{O}(t^{-4/3}),\quad t\to\infty,
\end{equation}
\end{theorem}
\begin{proof}
It suffices to show that
\begin{equation}\label{HP-n1}
|\left[ \E^{-\I t H}P_c \right]_{n,k}|\le C (1+|n|)(1+|k|)t^{-4/3},\quad t\ge 1.
\end{equation}
The representation \eqref{fh} and the bounds \eqref{est3}--\eqref{estC} imply
\begin{equation} \label{alg-dif}
h^\pm_n(\theta),~~\frac{\partial}{\partial \theta} h^\pm_n(\theta)
\in \mathcal A \quad \mbox{if} \quad q\in\ell^1_2.
\end{equation}
Therefore,
$\frac{d}{d\theta} W(\theta):=W'(\theta)\in \mathcal A$. Since
in the non-resonant case $W(\theta)^{-1}\in\mathcal A$ we also infer
\begin{equation}\label{imp12}
\frac{d}{d\theta}T(\theta),\quad \frac{d}{d\theta}  R^\pm(\theta)\in\mathcal A
\end{equation}
by Wiener's lemma. For the derivatives of $h_k^\pm$ bounds of the type \eqref{est10} hold, namely,
\begin{equation}\label{est11}
\Vert{\frac{\partial}{\partial \theta} h^\pm_n}(\cdot)\Vert_{\mathcal A}
\le \tilde C ~~{\rm for}~~ \pm n>0.
\end{equation}
For $n\le k$ we represent the jump of the resolvent across the spectrum  as
\[
\calR(\omega+\I 0)- \calR(\omega-\I 0))=
\frac{ T(\theta)f_k^+(\theta) f_n^-(\theta) + \overline{T(\theta)f_k^+ (\theta)} \overline{f_n^- (\theta)}}{-2\I\sin\theta},
\quad \theta\in[0,\pi].
\]
The scattering relations \eqref{scat-rel} imply
\[
 f_n^-(\theta)=T(-\theta) f_n^+(-\theta)-R^-(-\theta)f_n^-(-\theta),\quad
\overline{f_k^+(\theta)}=
T(\theta) f_k^-(\theta)-R^+(\theta)f_k^+(\theta).
\]
Then using the consistency relation $T \overline { R^- }+ \overline {T } R^+=0$ we come to the formula (cf. \cite[p.13]{S})
\[
\calR(\omega+\I 0)- \calR(\omega-\I 0))=\frac{|T(\theta)|^2}{-2\I\sin\theta}
[f_k^+(\theta)f_n^+(-\theta)+f_k^-(\theta)f_n^-(-\theta)],\quad
\theta\in[0,\pi].
\]
Inserting this into \eqref{spR} and integrating by parts we get
\begin{align}\nonumber
\left[ \E^{-\I t H}P_c \right]_{n,k}
& = \frac{1}{\pi} \int_{-\pi}^{\pi}
\E^{-\I t (2 - 2\cos \theta)}|T(\theta)|^2
[f_k^+(\theta)f_n^+(-\theta)+f_k^-(\theta)f_n^-(-\theta)]d\theta\\
\nonumber
&=\frac{\I}{2\pi t}\int_{-\pi}^{\pi}\E^{-\I t (2 - 2\cos \theta)}
\frac{d}{d\theta}\Big[\frac{|T(\theta)|^2}{\sin\theta}
(f_k^+(\theta)f_n^+(-\theta)+f_k^-(\theta)f_n^-(-\theta))\Big]d\theta\\
\nonumber
&=\left[ \E^{-\I t H}P_c \right]_{n,k}^++ \left[ \E^{-\I t H}P_c \right]_{n,k}^-.
\end{align}
Evaluating the derivative we further obtain
\begin{align}\nonumber
\left[ \E^{-\I t H}P_c \right]_{n,k}^\pm
&=\frac{\I}{2\pi t}\int_{-\pi}^{\pi}\E^{-\I t (2 - 2\cos \theta)}
\frac{d}{d\theta}\Big[\frac{|T(\theta)|^2}{\sin\theta}\E^{\mp\I\theta(k-n)}h _k^\pm(\theta)h_n^\pm(-\theta)\Big]d\theta\\
\nonumber
&=\frac{\pm(k-n)}{2\pi t}\int_{-\pi}^{\pi}\E^{-\I t (2 - 2\cos \theta)}\E^{\mp\I\theta(k-n)}
\frac{|T(\theta)|^2}{\sin\theta}h _k^\pm(\theta)h_n^\pm(-\theta)d\theta\\
\label{I-est}
&-\frac{\I}{2\pi t}\int_{-\pi}^{\pi}\E^{-\I t (2 - 2\cos \theta)}\E^{\mp\I\theta(k-n)}\cos\theta
\frac{|T(\theta)|^2}{\sin^2\theta}h _k^\pm(\theta)h_n^\pm(-\theta)d\theta\\
\nonumber
&+\frac{\I}{2\pi t}\int_{-\pi}^{\pi}\E^{-\I t (2 - 2\cos \theta)}\E^{\mp\I\theta(k-n)}
\frac{\frac{d}{d\theta}\Big[|T(\theta)|^2h _k^\pm(\theta)h_n^\pm(-\theta)\Big]}{\sin\theta}d\theta.
\end{align}
Next, observe that formula \eqref{est3} implies that
if $q\in \ell_2^1$, then $B^\pm_{m,s}\in \ell_1^1(\Z_\pm)$ for any fixed $m$,
and consequently
\begin{equation}\label{proper4}
S^\pm_m(j):=\sum_{s=j}^{\pm\infty}|B^\pm_{m,s}|, \quad S^\pm_m(\cdot)\in\ell^1(\Z_\pm).
\end{equation}
Based on this observation we prove the following
\begin{lemma}\label{T-est}
Let $q\in \ell^1_2$ and  $W(0)W(\pi)\neq 0$. Then
$T(\theta)h _m^\pm(\theta)/\sin\theta\in \mathcal A$,  and
\begin{equation}\label{Test1}
\Big\Vert\frac{T(\theta)h _m^\pm(\theta)}{\sin\theta}\Big\Vert_{\mathcal A}\le C (1 + |m|),\quad m\in \Z.
\end{equation}
\end{lemma}
\begin{proof}
Since  $T(\theta)/\sin\theta=2\I/W(\theta)$ then for $m\in\Z_\pm$ the
bound \eqref{Test1} follows from \eqref{est10} and Theorem~\ref{thm:scat}.
Hence it remains to consider the case $m\in \Z_\mp$. The scattering relations \eqref{scat-rel} imply
\begin{align}\nonumber
T(\theta)h _m^\pm(\theta)
= &(R^\mp(\theta)+1)h_m^\mp(\theta)\E^{\pm 2\I m \theta}-
(h_m^\mp(\theta)-h_m^\mp(-\theta))\E^{\pm 2\I m\theta} \\\label{proper8}
&  + h_m^\mp(-\theta)(1 -\E^{\pm 2\I m\theta}).
\end{align}
Using \eqref{fh} we obtain
\begin{align*}
\frac{h_m^\mp(\theta)-h_m^\mp(-\theta)}{\sin\theta}&=
\sum\limits_{s=\mp 1}^{\mp\infty}
B^\mp_{m,s}\frac{\E^{\mp\I s\theta}-\E^{\pm\I s\theta}}{\sin\theta}=\mp 2\I\sum\limits_{s=\mp 1}^{\mp \infty}
B^\mp_{m,s} \!\!\sum_{j=-(s-1)}^{s-1} \!\! \frac{1-(-1)^{s+j}}{2}\E^{\I j\theta}\\
&=\mp 2\I\sum\limits_{j=-\infty}^{\infty}\Big(\sum\limits_{s=\mp |j|\mp 1}^{\mp\infty}
\frac{1-(-1)^{s+j}}{2} B^\mp_{m,s}\Big)\E^{\I j\theta}.
\end{align*}
Property \eqref{proper4} then  implies
\begin{equation}\label{proper5}
\Big\Vert\frac{h_m^\mp(\theta)-h_m^\mp(-\theta)}{\sin\theta}\Big\Vert_{\mathcal A}\le C ,\quad m\in\Z_\mp.
\end{equation}
and we get
\[
\frac{f_0^\mp(\theta)-f_0^\mp(-\theta)}{\sin\theta},~
\frac{f_{1}^+(\theta)-f_{1}^+(-\theta)}{\sin\theta},~
\frac{f_{-1}^-(\theta)-f_{-1}^-(-\theta)}{\sin\theta}
\in \mathcal A,\quad q\in\ell_2^1,
\]
as well as
\begin{equation}\label{proper9}
\frac{R^\mp(\theta)+1}{\sin\theta}=\frac{1}{W(\theta)}\frac{W(\theta)\mp W^\mp(\theta)}{\sin\theta }\in\mathcal A.
\end{equation}
Furthermore,
\begin{equation}\label{proper10}
\Big\Vert\frac{1 -\E^{\pm 2\I m\theta}}{\sin\theta}\Big\Vert_{\mathcal A}\leq  2|m|.
\end{equation}
Finally, substituting \eqref{proper5}, \eqref{proper9}, and \eqref{proper10} into \eqref{proper8} we get \eqref{Test1}.
\end{proof}
To obtain \eqref{HP-n1} for the first summand in \eqref{I-est} note that $k-n\le 2\max\{|n|,|k|\}$.
Hence  we  apply \eqref{Test1}
to the factor $T(-\theta)h _m^\pm(-\theta)/\sin\theta$, where $|m|= \min\{|n|,|k|\}$.
Then we split the domain of integration into regions where
either the second or third derivative of the phase is nonzero and apply Lemma~\ref{lem:vC} together with
estimate from Theorem~\ref{thm:scat} and Lemma~\ref{T-est}.
To obtain \eqref{HP-n1} for the second  summand in \eqref{I-est} we apply \eqref{Test1}
to both $T(-\theta)h _n^\pm(-\theta)/\sin\theta$ and $T(\theta)h _k^\pm(\theta)/\sin\theta$.

To complete the proof of \eqref{Schr-asn1} we need one more property.
\begin{lemma}\label{lem:proizv} Let $q\in\ell_2^1$ and $W(0)W(\pi)\neq 0$. Then $\frac{d}{d\theta}
(T(\theta)h_m^\pm(\theta))\in\mathcal A$ with
\begin{equation}\label{proper12}
\Big\Vert \frac{d}{d\theta}(T(\theta)h_m^\pm(\theta))\Big\Vert_{\mathcal A}\leq C(1+|m|),\quad m\in\Z.
\end{equation}
\end{lemma}
\begin{proof} Since $T^\prime(\theta)$ and $\frac{d}{d\theta}h_m^\pm(\theta)$
are elements of $\mathcal A$ for $q\in\ell_2^1$,
then for $m\in \mathbb Z_\pm$ the statement of the Lemma is evident in view of \eqref{est11}.
To get it for $m\in \mathbb Z_\mp$ we use \eqref{imp12}, \eqref{est11}, and formula
\[
\frac{d}{d\theta}(T(\theta)h_m^\pm(\theta))=\frac{d}{d\theta}\left(R^\mp(\theta)h_m^\mp(\theta)\right)\,\E^{\pm 2\I m\theta} \pm 2\I m \E^{\pm 2\I m\theta}R^\mp(\theta)h_m^\mp(\theta) +\frac{d}{d\theta}h_m^\mp(-\theta).
\]
\end{proof}
The bound \eqref{HP-n1} for the third summand in \eqref{I-est} now follows
combining Theorem~\ref{thm:scat}, Lemmas~\ref{lem:vC}, \ref{T-est}, and \ref{lem:proizv}.
\end{proof}

\begin{theorem}\label{t-new1}
Let $q\in\ell^1_2$. Then in the non-resonant case the asymptotics
\eqref{as1-new} hold, i.e.,
\begin{equation}\label{Schr-asn2}
\Vert \E^{-\I tH} P_c\Vert_{\ell^2_\sigma\to \ell^2_{-\sigma}}=\mathcal{O}(t^{-3/2}),\quad t\to\infty,\quad\sigma>3/2.
\end{equation}
\end{theorem}
\begin{proof}
We will derive \eqref{Schr-asn2} for $[\E^{-\I tH} P_c]^+$, defined in \eqref{I-est}.
For $[\E^{-\I tH} P_c]^-$ the proof is similar.
Abbreviate
\begin{equation}\label{Psi-nk}
Z_{n,k}(\theta)=|T(\theta)|^2h _k^+(\theta)h_n^+(-\theta).
\end{equation}
Due to \eqref{I-est} it suffices to consider the operators $\cM_j(t)$, $j=1,2,3$,
with the kernels
\[
[\cM_j(t)]_{n,k}=\int_{-\pi}^{\pi}\E^{-\I t\phi_v(\theta)}Z_{n,k}^j(\theta) d\theta,\quad
\phi_v(\theta)=2-2\cos\theta+v\theta,\quad v=|k-n|/t,
\]
where
\begin{equation}\label{Z-nk}
Z_{n,k}^1(\theta)=\frac{k-n}{\sin\theta}Z_{n,k}(\theta),~~
Z_{n,k}^2(\theta)=\frac{\cos\theta}{\sin^2\theta}Z_{n,k}(\theta),~~
Z_{n,k}^3(\theta)=\frac{\frac{d}{d\theta}Z_{n,k}(\theta)}{\sin\theta}
\end{equation}
and obtain the bound
\begin{equation}\label{M-est}
\sum\limits_{n,k\in\Z}[\cM_j(t)]^2_{n,k}
\frac{1}{(1+|n|)^{2\sigma}(1+|k|)^{2\sigma}}\le Ct^{-1}
\end{equation}
for any $\sigma>3/2$ and sufficiently large $t\ge 1$.
As in the proof of Theorem~\ref{end1} step ii) we consider the integrals over $J:=\{\theta:\ |\theta\pm \pi/2|\leq\pi/6\}$
and over $[-\pi,\pi]\setminus J$.

For the integrals over $[-\pi,\pi]\setminus J$ we apply Lemma~\ref{lem:vC} with $s=2$ together with the fact
that $\Vert Z_{n,k}^j\Vert_{\mathcal A}\le C(1+|n|)(1+|k|)$, $j=1,2,3$ and obtain
\[
\Big|\int\limits_{[-\pi,\pi]\setminus J}\E^{-\I t\phi_v(\theta)}Z_{n,k}^j(\theta) d\theta\Big|
\le Ct^{-1/2}(1+|n|)(1+|k|),\quad t\ge 1.
\]
Then the bound \eqref{M-est} for
\[
[\tilde\cM_j(t)]_{n,k}=\int\limits_{[-\pi,\pi]\setminus J}\E^{-\I t\phi_v(\theta)}Z_{n,k}^j(\theta) d\theta
\]
follows.
To estimate the integrals over $J$ we consider the three possibilities
(a) $n\le k \le 0$, (b) $ 0\le n\le k$ and (c) $n\le 0\le k$.

Consider the case (b). Lemma \ref{newl} (ii) for $s=2$ imply
\begin{equation}\label{dTR-est1}
|\frac{d^2}{d\theta^2}T(\theta)|,\quad|\frac{d^2}{d\theta^2}R^\pm(\theta)|\le C,\quad \theta\in J.
\end{equation}
Respectively,
\[
|\frac{\partial}{\partial \theta}Z_{n,k}^j(\theta)|\le C(1+k),\quad\theta\in J,\quad 0\le n\le k.
\]
Then we obtain \eqref{M-est} by the same arguments as for the proof of Theorem~\ref{end1} step ii).\\

Consider the case (c).
Using the scattering relations \eqref{scat-rel} and
equality \eqref{nk1} we obtain
\begin{align*}
[\cM_j^{\pm}(t)]_{n,k}:&=\int_{J}
\E^{-\I t\phi_v(\theta)}Z_{n,k}^j(\theta) d\theta\\
&=\int_{J}
\E^{-\I t\phi_v(\theta)}Z_{n,k,1}^j(\theta) d\theta
+\int_{J}
\E^{-\I t{\tilde\phi}_v(\theta)}Z_{n,k,2}^j(\theta) d\theta
\end{align*}
where ${\tilde\phi}_v(\theta)$ is defined in  \eqref{tphi}, and
\begin{align*}
Z_{n,k,j}^1(\theta)&=\frac{k-n}{\sin\theta}Z_{n,k,j}(\theta),~~
Z_{n,k,j}^2(\theta)=\frac{\cos\theta}{\sin^2\theta}Z_{n,k,j}(\theta),\\
Z_{n,k,1}^3(\theta)&=\frac{\frac{d}{d\theta}Z_{n,k,1}(\theta)}{\sin\theta},~~
Z_{n,k,2}^3(\theta)=\frac{\frac{d}{d\theta}Z_{n,k,2}(\theta)-2\I n Z_{n,k,2}}{\sin\theta}
\end{align*}
with
\[
Z_{n,k,1}(\theta)=T(\theta)h _k^+(\theta)h_n^-(\theta), \quad
Z_{n,k,2}(\theta)=T(\theta)R^-(-\theta)h _k^+(\theta)h_n^-(-\theta).
\]
Lemma \ref{newl} (ii) and \eqref{dTR-est1} imply
\[
|\frac{\partial}{\partial \theta}Z_{n,k,1}^j(\theta)|, \:
|\frac{\partial}{\partial \theta}Z_{n,k,2}^j(\theta)|
\le C(1+\max\{|n|,|k|\}),~~\theta\in J.
\]
Hence \eqref{M-est} for the case (c) also follows.

It remains to consider the case (a).
Denote
\[
{\hat\phi}_v(\theta)=2-2\cos\theta+\hat v\theta,~~{\rm where}~~
\hat v(\theta)=-|k-n|/t\le 0.
\]
The scattering relations \eqref{scat-rel}
now imply
\begin{align*}
[\cM_j^{\pm}(t)]_{n,k}:&=\int_{J}
\E^{-\I t\phi_v(\theta)}Z_{n,k}^j(\theta) d\theta
=\int_{J}
\E^{-\I t\phi_v(\theta)}X_{n,k,1}^j(\theta) d\theta\\
&+\int_{J}
\E^{-\I t{\tilde\phi}_v(\theta)}X_{n,k,2}^j(\theta) d\theta
+\int_{J}
\E^{-\I t{\hat\phi}_v(\theta)}X_{n,k,3}^j(\theta) d\theta,
\end{align*}
where
\begin{align*}
X_{n,k,j}^1(\theta)&=\frac{k-n}{\sin\theta}X_{n,k,j}(\theta),\quad
X_{n,k,j}^2(\theta)=\cos\theta~\frac{X_{n,k,j}(\theta}{\sin^2\theta},\\
X_{n,k,1}^3(\theta)&=\frac{\frac{d}{d\theta}X_{n,k,1}(\theta)}{\sin\theta},
X_{n,k,3}^3(\theta)=\frac{\frac{d}{d\theta}X_{n,k,3}(\theta)+2\I (k-n)X_{n,k,3}}{\sin\theta}\\
X_{n,k,2}^3(\theta)&=\frac{\frac{d}{d\theta}X_{n,k,2}(\theta)+2\I kR^-(\theta)h _k^-(\theta)h_n^-(\theta)
-2\I\, n R^-(-\theta)h _k^-(-\theta)h_n^-(-\theta)}{\sin\theta}
\end{align*}
and
\begin{align*}
X_{n,k,1}(\theta)&=h_k^-(-\theta)h_n^-(\theta),\\
X_{n,k,2}(\theta)&=R^-(-\theta)h _k^-(-\theta)h_n^-(-\theta)+R^-(\theta)h _k^-(\theta)h_n^-(\theta),\\
X_{n,k,3}(\theta)&=|R^-(\theta)|^2h _k^-(\theta)h_n^-(-\theta).
\end{align*}
From Lemma \ref{newl} (ii) and \eqref{dTR-est1} it follows that
\[
|\frac{\partial}{\partial \theta}X_{n,k,m}^j(\theta)|
\le C(1+|n|)(1+|k|),~~\theta\in J.
\]
The integrals with the phase functions $\phi_v(\theta)$ and ${\tilde\phi}_v(\theta)$ can been estimated
similarly as in the previous cases. Since $\hat v\le 0$, then in order
to estimate the integrals with the phase functions ${\hat\phi}_v(\theta)$ we can interchange the methods
for $|\theta- \pi/2|\le \pi/6$ and for $|\theta+\pi/2|\le \pi/6$.
\end{proof}

\setcounter{equation}{0}

\section{Wave equation}
\label{KG-sect}
Here we extend our main results to the  wave equation \eqref{KGE}.
\subsection{Free wave equation}
\label{KGF-sect}
Set $\mathbf{u}_n(t)= \bigl(u_n(t),\dot u_n(t)\bigr)$.
Then \eqref{KGE} with $q=0$ reads
\begin{equation}\label{KGE0}
  \I\dot{\mathbf{u}}(t)=\mathbf{H}_0 \mathbf{u}(t),\quad t\in\R,
\end{equation}
where
\begin{equation*}
\mathbf{H}_0=
\begin{pmatrix}
  0               &   \I\\
 \I(\Delta_L-\mu^2)   &   0
\end{pmatrix}.
\end{equation*}
The continuous spectrum of $\mathbf{H}_0$ coincides with $\overline\Gamma$, where
\[
\Gamma=(-\sqrt{\mu^2+4},-\mu)\cup(\mu,\sqrt{\mu^2+4}).
\]
The resolvent $\mathbf{R}_0(\omega)=(\mathbf{H}_0-\omega)^{-1}$
can be expressed in terms of  $\calR_0(\omega)=(H_0-\omega)^{-1}$ (see \cite{kkk}):
\begin{equation}\label{KGER0}
   \mathbf{R}_0(\omega)=
   \begin{pmatrix}
            \omega \calR_0(\omega^2-\mu^2)           &   \I \calR_0(\omega^2-\mu^2)\\
         -\I(1 +\omega^2 \calR_0(\omega^2-\mu^2))      &   \omega \calR_0(\omega^2-\mu^2)
     \end{pmatrix}.
\end{equation}
Note that factorizing $H_0$ according to $H_0-\mu^2=A^*A$ (cf.\ \cite[Sect.~11.1]{tjac}) we can
use $\bigl(u_n(t), A \dot u_n(t)\bigr)$ to write \eqref{KGE0} in self-adjoint form. We refer to \cite{GGHT}
and the references therein for further details.

Denote by ${\bf l}^p_\sigma=\ell^p_\sigma\oplus\ell^p_\sigma$, $p\ge 1$, $\sigma\in\R$.
\begin{lemma}\label{KGfree-as}
Let $\mu>0$. Then the following asymptotics hold
\begin{equation}\label{KGH0-as1}
\Vert \E^{-\I t\mathbf{H}_0}\Vert_{{\bf l}^1\to {\bf l}^\infty}=\mathcal{O}(t^{-1/3}),\quad t\to\infty,
\end{equation}
\begin{equation}\label{KGH0-as2}\Vert \E^{-\I t\mathbf{H}_0}\Vert_{{\bf l}^2_\sigma\to{\bf l}^2_{-\sigma}}
=\mathcal{O}(t^{-1/2}),\quad t\to\infty, \quad\sigma>1/2.
\end{equation}
\end{lemma}
\begin{proof}
As in the proof of Proposition \ref{free-as} we consider $t\ge 1$ and apply the spectral representation:
\[
\E^{-\I t\mathbf{H}_0}=\frac 1{2\pi \I}\int\limits_{\Gamma}
\E^{-\I t\omega}({\bf R_0}(\omega+\I 0)- \mathbf{R}_0(\omega-\I 0))\,d\omega.
\]
We prove asymptotics \eqref{KGH0-as1} and \eqref{KGH0-as2} only for the entry $[\E^{-\I t\mathbf{H}_0}]^{12}$.
The other entries of the matrix $\E^{-\I t\mathbf{H}_0}$ can be treated similarly.

Let $\theta_+=\theta_+(\omega^2-\mu^2)\in [-\pi,0]$ be the solution of
$2-2\cos\theta=\omega^2-\mu^2$ and $\theta_-=-\theta_+$.
Due to \eqref{R02} we have
\begin{equation}\label{KGPPf}
[\E^{-\I t\mathbf{H}_0}]^{12}_{n,k}=\frac 1{2\pi}\int\limits_{\Gamma}\E^{-\I t\omega}
\Big(\frac{\E^{-\I\theta_+|n-k|}}{\sin\theta_+}-
\frac{\E^{-\I\theta_-|n-k|}}{\sin\theta_-}\Big)d\omega=I_- + I_+,
\end{equation}
where
\begin{equation}\label{KGW}
I_\pm:=-\frac 1{2\pi}\int\limits_{-\pi}^{\pi}\frac{\E^{\pm\I t\,g(\theta)-\I\theta|n-k|}d\theta}{g(\theta)}
,\quad g(\theta):=\sqrt{2-2\cos\theta+\mu^2}.
\end{equation}
{\it Step i)}
To prove \eqref{KGH0-as1} it suffices  to obtain the bound
\begin{equation}\label{K-as}
\sup_{n,k}|[\E^{-\I t\mathbf{H}_0}]^{12}_{n,k}|\le Ct^{-1/3},\quad t\ge 1.
\end{equation}
We consider only the integral $I_-$.
Abbreviate $v:=\frac{|n-k|}t\ge 0$ and set $\varkappa=(2+\mu^2-\sqrt{4\mu^2+\mu^4})/2$,
$0<\varkappa<1$.
It is easy to check that if $v\not= v_0:=\sqrt\varkappa$ then
the phase function
\begin{equation}\label{phi}
\Phi_v(\theta)=g(\theta)+v\theta,
\end{equation}
where $g(\theta)$ is defined in \eqref{KGW}, has at most two non-degenerate stationary points.
In the case  $v=v_0$ there exists a unique degenerate stationary point $\theta_0=-\arccos\varkappa$,
$-\pi/2<\theta_0<0$,
such that  $\Phi'''(\theta_0)=\sqrt{\varkappa}\not =0$. Moreover, 
$g(\theta)>\mu>0$,  and therefore $g^{-1}(\theta)$ is a smooth function.
Hence \eqref{K-as} follows from the van der Corput lemma.\\
{\it Step ii)}
To prove \eqref{KGH0-as2}
we divide the domain of integration in $I_-$ into the domains 
$J=\{\theta:\:|\theta-\theta_0|\le \nu|\theta_0|\}$ and $[-\pi, \pi]\setminus J$,
where $\nu=\nu(\mu)$, $0<\nu\le 1$, will be specified below.
We further divide the domain $J$ into subdomains  $t_j\le |\theta-\theta_0|\le t_{j+1}$,
$0\le j\le N$, where $t_j$ for $j=1,\dots,N$ is chosen as in Lemma~\ref{Lem:K}, and $t_{N+1}=\nu|\theta_0|$.
The asymptotics \eqref{KGH0-as2} for the part over $[-\pi, \pi]\setminus J$ follow from the stationary phase method.
To get \eqref{KGH0-as2} for
\[
[{\mathbf K}_j(t)]_{n,k}=\int\limits_{t_j\le |\theta-\theta_0|\le t_{j+1}}\E^{-\I t\Phi_v(\theta)}
\frac{d\theta}{g(\theta)},\quad 1\le j\le N,
\]
we consider $|v_0-v|\le \frac 12 v_0t_jt^\ve$ and $|v_0-v|\ge \frac 12 v_0t_jt^\ve$ separately.
The first case is identical to the first case of Lemma \ref{Lem:K}.
In the second case we apply integration by parts similarly to \eqref{denom}. Namely, we have to estimate: (a) $|\Phi^\prime_v(\theta)|^{-1}$ at the points $\theta_0\pm t_j$ and $\theta_0 \pm t_{j+1}$,  and (b) the integral of the function $|\Phi^{\prime\prime}_v(\theta)|(\Phi^\prime_v(\theta))^{-2}$ between these points. But since the function
$\Phi^{\prime\prime}_v(\theta)$ does not change its sign on the intervals
$[\theta_0+t_j, \theta_0+t_{j+1}]$ and $[\theta_0-t_{j+1},\theta_0-t_j]$, then the antiderivative of the function
$|\Phi^{\prime\prime}_v(\theta)|(\Phi^\prime_v(\theta))^{-2}$ is equal up to a sign to the function $(\Phi^\prime_v(\theta))^{-1}$. Thus, it is sufficient to consider the case (a) only.

We have ${\Phi}_v(\theta)=g(\theta)+v\theta$, therefore
\[
\Phi^\prime_v(\theta)=g^\prime(\theta)+v
=g^\prime(\theta_0)+\frac 12 g^{\prime\prime\prime}(\tilde\theta)(\theta - \theta_0)^2+v
=\frac 12 g^{\prime\prime\prime}(\tilde\theta)(\theta - \theta_0)^2+v-v_0.
\]
Here we used formulas $g^{\prime}(\theta_0)=-v_0$ and $g^{\prime\prime}(\theta_0)=0$.
Hence for large $t$
\[
|\Phi^\prime_v(\theta_0\pm t_{j+s})|\ge |v-v_0|-Ct^2_{j+s}\ge t_j(\frac 12 v_0t^\ve-C)\ge C_1t_j, \quad j=1,..., N-1,\quad s=0,1,
\]
and then
\[
|[{\mathbf K}_j(t)]_{n,k}|\le Ct^{-1}t_j^{-1}\le C t^{-1/2}, \quad j=1,..., N-1
\]
as in \eqref{dop14}.

In the case $j=N$ we have $|v-v_0|\ge \frac 12 v_0$. Further,
\[
\Phi^\prime_v(\theta_0\pm t_{N+1})
=\frac 12 g^{\prime\prime\prime}(\tilde\theta)(\nu|\theta_0|)^2+v-v_0
\]
Since $|g^{\prime\prime\prime}(\theta)|\le G=G(\mu)$, $\theta\in [-\pi,\pi]$,
then we can choose 
$\nu=\min\{1, \sqrt{\frac{2v_0}{3G\theta_0^2}}\,\}$ to obtain
$|\Phi^\prime_v(\theta_0\pm t_{N+1})|\ge \frac 16 v_0$.
Respectively, $|\Phi^\prime_v(\theta_0\pm t_{N+1})|^{-1}\le 6/v_0$, and hence
\[
|[{\mathbf K}_N(t)]_{n,k}|\le Ct^{-1}.
\]
\end{proof}
\begin{remark}
The solution of the free wave equation \eqref{KGE0}, corresponding to $\mu=0$,  does not decay
as $t\to\pm\infty$. In fact, the first component of the solution  is given by
\begin{equation}
u_n(t) = \sum_{m\in\Z} c_{n-m}(t) u_m(0) + s_{n-m}(t) \dot{u}_m(0),
\end{equation}
where
\begin{align}
c_n(t) &= \frac{1}{2\pi} \int_{-\pi}^\pi \cos(\sqrt{1-\cos\theta} \sqrt{2}t)
\mathrm{e}^{\I \theta n} d\theta = J_{2|n|}(2 t),\\ \nonumber
s_n(t) &= \frac{1}{2\pi} \int_{-\pi}^\pi \frac{\sin(\sqrt{1-\cos\theta} \sqrt{2}t) }{\sqrt{1-\cos\theta}}
\mathrm{e}^{\I \theta n} d\theta = \int_0^t c_n(s) ds\\
&= \frac{t^{2|n|+1}}{2^{|n|} (|n|+1)!} \,
{}_1F_2\Big(\frac{2|n|+1}{2};(\frac{2|n|+3}{2},2|n|+1); -t^2\Big).
\end{align}
Here $J_n(x)$, ${}_pF_q(\underline{u};\underline{v};x)$ denote the Bessel and generalized hypergeometric functions, respectively.
In particular, while $c_n(t)=O(t^{-1/2})$ for fixed $n$, we have $s_n(t)=\frac{1}{2} + O(t^{-1/2})$ for fixed $n$.
\end{remark}
\subsection{Perturbed  wave equation}
\label{KGP-sect}
In matrix form \eqref{KGE} reads
\begin{equation}\label{KGE1}
  \I\dot{\mathbf{u}}(t)= \mathbf{H} \mathbf{u}(t),\quad t\in\R,
\end{equation}
where
\[
\mathbf{H}=
\begin{pmatrix}
  0               &   \I\\
 \I(\Delta_L-\mu^2-q)   &   0
\end{pmatrix}.
\]
The resolvent $\mathbf{R}(\omega)=(\mathbf{H}-\omega)^{-1}$
can be expressed in terms of  $\calR(\omega)=(H-\omega)^{-1}$ (see \cite{kkk}):
\begin{equation}\label{KGE4}
 \mathbf{R}(\omega)=
   \begin{pmatrix}
            \omega \calR(\omega^2-\mu^2)           &   \I \calR(\omega^2-\mu^2)\\
         -\I (1 +\omega^2 \calR(\omega^2-\mu^2))      &   \omega \calR(\omega^2-\mu^2)
     \end{pmatrix}.
\end{equation}
Representation \eqref{KGE4} and Lemma \ref{BV} imply
the limiting  absorption principle for the perturbed  resolvent:
\begin{lemma}\label{Lw}
Suppose $q\in\ell^1$. Then
for $\omega\in\Gamma$ the convergence
\[
\mathbf{R}(\omega\pm \I\varepsilon)\to\mathbf{R}(\omega\pm \I 0),\quad\varepsilon\to 0+,
\]
holds in ${\mathcal L}({\bf l}^2_\sigma,{\bf l}^2_{-\sigma})$ with $\sigma>1/2$.
\end{lemma}
For the dynamical group associated with the perturbed wave equation \eqref{KGE1}
the spectral representation of type \eqref{spR} holds:
\begin{equation}\label{PP1}
   \E^{-\I t\mathbf{H}}{\bf P}_{c}
   =\frac 1{2\pi \I}\int\limits_{\Gamma}
   \E^{-\I t\omega}( \mathbf{R}(\omega+\I 0)- \mathbf{R}(\omega-\I 0))\,d\omega.
\end{equation}
Here ${\bf P}_{c}$ is the projection onto the continuous spectrum of $\mathbf{H}^2$.
Next, we prove asymptotics of type \eqref{fullp} and \eqref{as1-new} for \eqref{KGE1}.
\begin {theorem}\label{end3}
Let $\mu>0$ and $q\in\ell^1_1$. Then the following asymptotics holds
\begin{equation}\label{KGE-as1}
\Vert \E^{-\I t\mathbf{H}}{\bf P}_{c}\Vert_{{\bf l}^1\to {\bf l}^\infty}=
\mathcal{O}(t^{-1/3}),\quad t\to\infty.
\end{equation}
and
\begin{equation}\label{KGE-as2}
\Vert\E^{-\I t\mathbf{H}} {\bf P}_c\Vert_{{\bf l}^2_\sigma\to{\bf l}^2_{-\sigma}}
=\mathcal{O}(t^{-1/2}),\quad t\to\infty,\quad\sigma>1/2.
\end{equation}
\end{theorem}
\begin{proof}
{\it Step i)}
Due to the representation \eqref{PP1} and formula \eqref{RJ-rep} it suffices to
obtain \eqref{KGE-as1}--\eqref{KGE-as2} for the operator with the kernel
\begin{align}\nonumber
[\bK(t)]_{n,k} &=\int\limits_\mu^{\sqrt{\mu^2+4}}\!\!\!\E^{-\I t\omega}\left[
\frac{f_k^+(\theta_+) f_n^-(\theta_+)}{W(\theta_+)}
- \frac{f_k^+(\theta_-) f_n^-(\theta_-)}{W(\theta_-)} \right] d\omega\\
\label{Knk}
\!\!\!\!&=-\!\!\!\int\limits_{-\pi}^{\pi}\E^{-\I tg(\theta)}\,
\frac{f_k^+(\theta) f_n^-(\theta)}{W(\theta)}\,
\frac{\sin\theta\, d\theta}{g(\theta)}\\
\nonumber
\!\!\!\!&=\!\frac \I2\int\limits_{-\pi}^{\pi}\E^{-\I t\Phi_v(\theta)}\,\,
h_k^+(\theta) h_n^-(\theta)T(\theta)\,
\frac{d\theta}{g(\theta)},~~n\le k,
\end{align}
where the  phase function $\Phi_v(\theta)$ is defined in \eqref{phi}.
For \eqref{KGE-as1} we need to prove that
\begin{equation}\label{PP2}
\sup_{n\le k}\Big|[\mathbf{K}(t)]_{n,k}\Big|\le C t^{-1/3},\quad t\ge 1.
\end{equation}
Just as in the proof of Theorem \ref{end1} (i)
we consider three different cases (a), (b) and (c).
Using the properties of the  phase function obtained in Lemma~\ref{KGfree-as}
one can now proceed as in the proof of Theorem~\ref{end1} (i).\\
{\it Step ii)}
Recall that $-\pi/2<\theta_0<0$. Denote
${\bf J}:=\{\theta:~|\theta\pm\theta_0|\le \nu|\theta_0|\}$, where $\nu=\nu(\mu)$ 
is defined in Lemma \ref{KGfree-as}.
We represent $[\bK(t)]_{n,k}$ as the sum
\[
[\bK(t)]_{n,k}=[\bK^{\pm}(t)]_{n,k}+[{\tilde\bK}(t)]_{n,k}
\]
where
\begin{align*}
[\bK^\pm(t)]_{n,k}&=\frac{1}{2\pi} \int_{|\theta\pm\theta_0|\le \delta}
\E^{-\I t\Phi_v(\theta)} Y_{n,k}(\theta) \frac{d\theta}{g(\theta)},\\
[\tilde{\bK}(t)]_{n,k}&=\frac{1}{2\pi} \int_{\theta\in[-\pi,\pi]\setminus {\bf J}}
\E^{-\I t\Phi_v(\theta)} Y_{n,k}(\theta) \frac{d\theta}{g(\theta)},
\end{align*}
and $Y_{n,k}(\theta)=h_k^+(\theta)h_n^-(\theta) T(\theta)$ as above.
Note that the bound \eqref{phi-c} holds for $\theta\in{\bf J}$ also. Hence
one can now proceed as in the proof of Theorem~\ref{end1} (ii) to obtain \eqref{KGE-as1}.
\end{proof}

\begin{remark}
In the case $\mu=0$ the factor $\sin(\theta/2)$ in the denominator
of \eqref{Knk} implies that we cannot get \eqref{KGE-as1} and \eqref{KGE-as2}
for $[\E^{-\I t{\bf H}}{\bf P}_c]^{12}$ in this case.
Nevertheless, the analogous expression for the other entries of $[\E^{-\I t{\bf H}}{\bf P}]$
does not contain this factor in the denominator and hence
asymptotics \eqref{KGE-as1} and \eqref{KGE-as2} with the decay rate $t^{-1/3}$ hold for these entries.
\end{remark}
Now we consider the non-resonant case and obtain asymptotics of type \eqref{as-new}
and \eqref{as2-new} for  equation \eqref{KGE1}. To this end note that $\mathbf{H}$ has a resonance at a boundary
point of the continuous spectrum $\overline{\Gamma}$ if and only if $H$ has a resonance at the corresponding
boundary point of its continuous spectrum $[0,4]$.
\begin {theorem}\label{wend}
i) Let $\mu\ge0$, $q\in\ell^1_2$. Then in the non-resonant case the following asymptotics hold
\begin{equation}\label{KG-asn1}
\Vert \E^{-\I t{\bf H}}{\bf P}_{c}\Vert_{{\bf l}^1_1\to {\bf l}^\infty_{-1}}=\mathcal{O}(t^{-4/3}),\quad t\to\infty.
\end{equation}
ii) Let  $\mu>0$ and $q\in\ell^1_2$. Then in the non-resonant case and for any $\sigma>3/2$ the following asymptotics hold
\begin{equation}\label{KG-asn2}
\Vert \E^{-\I t{\bf H}}{\bf P}_c\Vert_{{\bf l}^2_\sigma\to {\bf l}^2_{-\sigma}}=\mathcal{O}(t^{-3/2}),\quad t\to\infty.
\end{equation}
iii) In the case  $\mu=0$ the asymptotics \eqref{KG-asn2} hold under the stronger conditions $q\in\ell^1_3$ and $\sigma>5/2$.
\end{theorem}
\begin{proof}
We consider
$[\E^{-\I t{\bf H}}{\bf P}_c]^{12}$ and the case $n\le k$ only.
As in the proof of Theorems \ref{t-new} and \ref{t-new1} we need to consider the operators $\bM_j(t)$ with the kernels
\[
[\bM_j(t)]_{n,k}=t^{-1}\int_{-\pi}^{\pi}\E^{-\I t\Phi_v(\theta)}Z_{n,k}^j(\theta) d\theta,
\]
where $Z_{n,k}$ and $\Phi_v(\theta)$ are defined in \eqref{Psi-nk}--\eqref{Z-nk} and \eqref{phi},
and obtain the asymptotics
\begin{equation}\label{M-asn1}
\Vert \bM_j(t)\Vert_{{\bf l}^1_1\to {\bf l}^\infty_{-1}}=\mathcal{O}(t^{-4/3}),\quad t\to\infty,
\end{equation}
\begin{equation}\label{M-asn2}
\Vert \bM_j(t)\Vert_{{\bf l}^2_\sigma\to {\bf l}^2_{-\sigma}}=\mathcal{O}(t^{-3/2}),\quad t\to\infty,\quad\sigma>3/2.
\end{equation}
The asymptotics {\eqref{M-asn1}} for $\mu\ge 0$ and the asymptotics {\eqref{M-asn2}} in the case $\mu>0$
can be established as in the proofs of Theorems \ref{t-new} and \ref{t-new1}.

In the case $\mu=0$ we have $\theta_0=0$ and ${\bf J}={\bf J}_0=\{\theta:~|\theta|\le \pi/6\}$.
The integrals over $[-\pi,\pi]\setminus{\bf J}_0$ can be estimated as in the case $\mu>0$.
Concerning the integral over ${\bf J}_0$ we will not split it as in \eqref{I-est} but consider the whole integral
\[
[\bM(t)]_{n,k}=t^{-1}\int_{{\bf J}_0}\E^{-2\I t\sin(\theta/2)}\frac{d}{d\theta}
\Big(\E^{-\I\theta(k-n)}\frac{|T(\theta)|^2}{\sin\theta}h_k^+(\theta)h_n^{+}(-\theta)\Big) d\theta.
\]
We apply integration by parts once more and obtain the asymptotics of type \eqref{M-asn2}
for $\bM(t)$ with decay rate $t^{-2}$ and with $\sigma>5/2$ if we prove that
\[
\Big|\frac{d^2}{d\theta^2}\left(\frac{|T(\theta)|^2}{\sin^2\theta}h_k^+(\theta)h_n^{+}(-\theta)\sin\theta\right)\Big|
\le C(1+|n|^2)(1+|k|^2),\quad \theta\in{\bf J}_0.
\]
Since  $T(\theta)/\sin(\theta)=2\I /W(\theta)$ and its first and second derivatives are bounded
for $q\in\ell^1_3$, it suffices to prove that
\begin{equation}\label{M1}
\Big|\frac{d^2}{d\theta^2} \big( h_k^+(\theta)h_n^{+}(-\theta)\sin\theta\big)\Big|
\le C(1+|n|^2)(1+|k|^2),\quad \theta\in{\bf J}_0,
\end{equation}
which follows from \eqref{Jostderiv} with $p=0,1$ and from the following bound
\[
|\frac{d^2}{d\theta^2} \big( h_m^{+}(\theta)\sin\theta\big)|\le C\max\{-m^2,1\},\quad \theta\in {\bf J}_0.
\]
The last bound in the case $m\ge 0$ follows from \eqref{Jostderiv} with $p=2$. In the case $m<0$
one needs to apply the scattering relation as before.
\end{proof}
\bigskip
\noindent{\bf Acknowledgments.} We are very grateful to Hans Georg Feichtinger, Michael Goldstein, Wilhelm Schlag, and Andreas Seeger for helpful discussions, to Michael Goldberg for hints with respect to the literature, and to Markus Holzleitner for bringing some typos to our attention.
We are also indebted to the referee for challenging us to improve our results.



\begin{thebibliography}{99}

\bibitem{CT}
S. Cuccagna and M. Tarulli,
{\em On asymptotic stability of standing waves of discrete Schr\"odinger equation in $\mathbb{Z}$},
SIAM J. Math. Anal. {\bf 41} (2009), 861--885.

\bibitem{DT}
P. Deift and E. Trubowitz,
{\em Inverse scattering on the line},
Comm. Pure Appl. Math. {\bf 32} (1979), 121--251.

\bibitem{EG}
I. Egorova and L. Golinskii,
{\em On the location of the discrete spectrum for complex Jacobi matrices},
Proc. Amer. Math. Soc. {\bf 133} (2005), 3635--3641.

\bibitem{emt}
I. Egorova, J. Michor, and G. Teschl,
{\em Scattering theory with finite-gap backgrounds:  Transformation operators and characteristic properties of scattering data},
Math. Phys. Anal. Geom. {\bf 16} (2013), 111--136.

\bibitem{EKMT}
I. Egorova, E. Kopylova, V. Marchenko, and G. Teschl,
{\em Dispersion estimates for one-dimensional Schr\"odinger and Klein--Gordon equations revisited},
\arxiv{1411.0021}

\bibitem{GGHT}
F. Gesztesy, J. A. Goldstein, H. Holden, and G. Teschl,
{\em Abstract wave equations and associated Dirac-type operators},
Ann. Mat. Pura Appl. {\bf 191} (2012), 631--676.

\bibitem{G}
M. Goldberg,
{\em Transport in the one-dimensional Schr\"odinger equation},
Proc. Amer. Math. Soc. {\bf 135} (2007), 171--3179.

\bibitem{GS}
M. Goldberg and W. Schlag,
{\em Dispersive estimates for Schr\"odinger operators in dimensions one and three},
Commun. Math. Phys. {\bf 251} (2004), 157--178.

\bibitem{jeka}
A. Jensen and T. Kato,
{\em Spectral properties of Schr\"odinger operators and time-decay of the wave functions},
Duke Math. J. {\bf 46} (1979), 583--611.

\bibitem{KPS}
P. G. Kevrekidis, D. E. Pelinovsky, and A. Stefanov,
{\em Asymptotic stability of small bound states in the discrete nonlinear Schr\"odinger equation}
SIAM J. Math. Anal. {\bf 41} (2009), 2010--2030.

\bibitem{KT}
M. Keel and T. Tao,
{\em Endpoint Strichartz estimates},
Amer. J. Math. {\bf 120} (1998), 955--980.

\bibitem{kk10}
A. Komech and E. Kopylova,
{\em Dispersion Decay and Scattering Theory},
John Wiley and Sons, Hoboken, NJ, 2012.

\bibitem{kkk}
A. Komech, E. Kopylova, and M. Kunze,
{\em Dispersion estimates for 1D discrete Schr\"odinger and Klein-Gordon equations},
Appl. Anal. {\bf 85} (2006), no. 12,  1487--1508.

\bibitem {k09}
E. Kopylova,
{\em On the asymptotic stability of solitary waves in the discrete Schr\"odinger equation coupled to a nonlinear oscillator},
Nonlinear Anal. {\bf 71} (2009), no. 7--8, 3031--3046.

\bibitem{PS11}
D. Pelinovsky and A. Sakovich,
{\em Internal modes of discrete solitons near the anti-continuum limit of the dNLS equation},
 Physica D {\bf 240} (2011), 265--281.

\bibitem{PS}
D. Pelinovsky and A. Stefanov,
{\em On the spectral theory and dispersive estimates for a discrete Schr\"odinger equation in one dimension},
J. Math. Phys. {\bf 49}, (2008), 113501.

\bibitem{S}
W. Schlag,
{\em Dispersive estimates for Schr\"odinger operators: a survey},
in "Mathematical aspects of nonlinear dispersive equations", 255--285, Ann. of Math. Stud. {\bf 163}, Princeton Univ. Press, Princeton, NJ, 2007.

\bibitem{SK}
A. Stefanov and P. G. Kevrekidis,
{\em Asymptotic behavior of small solutions for the discrete nonlinear Schr\"odinger and Klein-Gordon equations},
Nonlinearity {\bf 18} (2005), 1841--1857.

\bibitem{St}
E. M. Stein, {\em Harmonic analysis: real-variable methods, orthogonality, and oscillatory integrals},
Princeton Math. Series {\bf 43}, Princeton University Press, Princeton, NJ, 1993.

\bibitem{tjac}
G. Teschl,
{\em Jacobi Operators and Completely Integrable Nonlinear Lattices},
Math. Surv. and Mon. {\bf 72},  Amer.\ Math.\ Soc., Rhode Island, 2000.

\bibitem{W}
G. N. Watson,
{\em A Treatise on the Theory of Bessel Functions},
Cambridge University Press, Cambridge, 1886.
\end{thebibliography}
\end{document}